\documentclass[a4paper,11pt]{amsart}

\usepackage{amsmath,amsfonts,amssymb,graphics}
\usepackage{amsthm}
\usepackage{pgfplots}
\usepackage{cite}
\usepackage{hyperref}
\usepackage{color}
\usepackage[normalem]{ulem}
\usepackage{enumitem}
\usepackage{bbm}
\usepackage{a4wide}
\usepackage{multirow}
\usepackage{booktabs}
\usepackage[foot]{amsaddr}

\newcommand \id{\mathbbm 1}

\newtheorem{theorem}{Theorem}[section]
\newtheorem{lemma}[theorem]{Lemma}
\newtheorem{proposition}[theorem]{Proposition}
\newtheorem{corollary}[theorem]{Corollary}
\newtheorem{assumption}[theorem]{Assumption}

\newtheorem{definition}[theorem]{Definition}
\newtheorem{example}[theorem]{Example}
\theoremstyle{remark}

\numberwithin{equation}{section}

\newcommand {\R} {\mathbb{R}}

\renewcommand {\P} {\mathbb{P}}
\newcommand {\C} {\mathbb{C}}

\newcommand {\Sc} {\mathcal{S}}
\newcommand {\RP} {\mathbb{RP}}
\newcommand {\Hc} {\mathcal{H}}
\newcommand {\E} {\mathbb{E}}

\newcommand {\Fc} {\mathcal{F}}
\newcommand {\Xb} {\mathbb{X}}
\newcommand {\Xc} {\mathcal{X}}
\newcommand {\Pc} {\mathcal{P}}
\newcommand {\Tc} {\mathcal{T}}
\newcommand {\Gc} {\mathcal{G}}
\newcommand {\Dc} {\mathcal{D}}
\newcommand {\Cc} {\mathcal{C}}
\newcommand {\Pb} {\mathbb{P}}

\newcommand {\Lc} {\mathcal{L}}

\def\brb#1{\left[#1\right]}

\graphicspath{{./pics/}}
\begin{document}

\title[RSW estimates for random polynomials]{Russo-Seymour-Welsh estimates for the \\ Kostlan ensemble of random polynomials}

\author{D. Beliaev}
\address{Mathematical Institute, University of Oxford}
\email{belyaev@maths.ox.ac.uk}
\author{S. Muirhead}
\address{Mathematical Institute, University of Oxford (currently at King's College London)}
\email{stephen.muirhead@kcl.ac.uk}
\author{I. Wigman}
\address{Department of Mathematics, King's College London}
\email{igor.wigman@kcl.ac.uk}


\date{\today}

\begin{abstract} We study the percolation properties of the nodal structures of random fields.
Lower bounds on crossing probabilities (RSW-type estimates) of quads by nodal domains or nodal sets 
of Gaussian ensembles of smooth random functions are established under the following assumptions: (i) sufficient symmetry; (ii) smoothness and non-degeneracy; (iii) local convergence of the covariance kernels; (iv) asymptotically non-negative correlations; and (v) uniform rapid decay of correlations.

The Kostlan ensemble is an important model of Gaussian homogeneous random polynomials. An application of our theory to the Kostlan ensemble yields
RSW-type estimates that are uniform with respect to the degree of the polynomials and quads of controlled geometry, valid on all relevant scales.
This extends the recent results on the local scaling limit of the Kostlan ensemble, due to Beffara and Gayet.

\end{abstract}
\maketitle

\section{Introduction}

\subsection{The Kostlan ensemble}
\label{s:kost}

The Kostlan ensemble of homogeneous degree-$n$ polynomials in $m+1\ge 2$ variables is the Gaussian random field $f_{n}: \R^{m+1}\rightarrow\R$  defined as
\begin{equation}
\label{eq:fn Kostlan def}
f_{n}(x)=f_{n;m}(x) = \sum\limits_{|J|=n}\sqrt{\binom{n}{J}}a_{J}x^{J},
\end{equation}
where  $J=(j_{0},\ldots,j_{m})$ is the multi-index, $|J|=j_{0}+\ldots+j_{m}$, $\binom{n}{J} = \frac{n!}{j_{0}!\cdot \ldots\cdot j_{m}!}$, and $\{a_{J}\}$ are i.i.d.\ standard Gaussian random variables. Since $f_n$ is homogeneous, it is also natural to view the Kostlan ensemble as the Gaussian random field on the unit $m$-dimensional sphere~$\mathbb{S}^m$ that is the restriction of~\eqref{eq:fn Kostlan def} to $\mathbb{S}^m$. The natural extension of \eqref{eq:fn Kostlan def} to $\C^{m+1}$ is known as the `complex Fubini-Study' ensemble.

\begin{figure}[t]
\label{fig:Kostlan}
\centering
\includegraphics[scale=0.75]{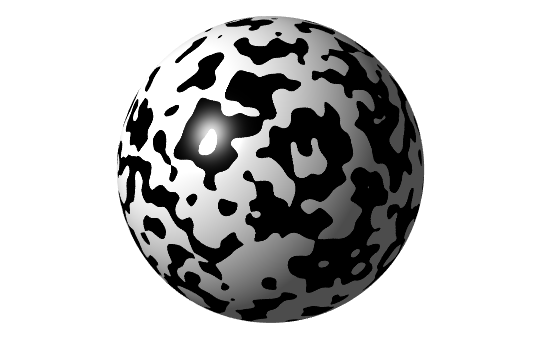}
\caption{A sample of the $m=2$-dimensional Kostlan ensemble of degree~$300$, with black (resp.\ white) representing the positive (resp.\ negative) nodal domains.}
\end{figure}

In this paper we are interested in various geometric properties of the {\em nodal set} of the Kostlan ensemble, i.e.\ the zero set of~$f_{n}$, particularly when the degree $n$ is large. Figure~\ref{fig:Kostlan} depicts the nodal domains of a sample of the $m=2$-dimensional Kostlan ensemble of degree~$300$ on~$\mathbb{S}^2$. Since $f_{n}$ is either even or odd depending on $n$, its nodal set can be naturally considered as a degree-$n$ \textit{hypersurface} (i.e.\ algebraic variety of co-dimension one) on the projective space~$\RP^m$. As we explain below, the Kostlan ensemble is a natural model for a `typical' homogeneous polynomial, and hence one may think of its nodal set as a `typical' real projective hypersurface.
\vspace{0.1cm}

The Kostlan ensemble can be equivalently defined as the canonical Gaussian element in the Hilbert space $\Hc_{n}$ of homogeneous degree-$n$ polynomials in $m+1$ variables spanned by the collection
\[ \left\{ \sqrt{\binom{n}{J}} x^{J}  \right\}_{|J|=n} \]
 as its orthonormal basis. Restricted to $\Hc_{n}$, the associated scalar product is,
up to the constant $\sqrt{n!}$, equal to the scalar product in the Bargmann-Fock space \cite{Bargmann61},
i.e.\ the space of all analytic functions on~$\C^{m+1}$ such that
\[
\|f\|_{BF}^2=\frac{1}{\pi^{m+1}}\int_{\C^{m+1}}|f(z)|^2e^{-\|z\|^2}d z<\infty
\]
with the scalar product
\begin{equation} \label{e:bfs}
\langle f,g \rangle_{BF}=\frac{1}{\pi^{m+1}}\int_{\C^{m+1}}f(z)\bar{g}(z)e^{-\|z\|^2}d z,
\end{equation}
playing an important role in quantum mechanics.
The restriction of the scalar product \eqref{e:bfs} to~$\Hc_{n}$ satisfies the following important property, relevant in our setting: it is the unique (up to a scale factor) scalar product on the space of degree-$n$ homogeneous polynomials on~$\C^{m+1}$ that is
invariant w.r.t.\ the unitary group. In other words, the Kostlan ensemble \eqref{eq:fn Kostlan def} is the real trace of the \textit{unique} unitary invariant Gaussian ensemble of homogeneous polynomials (although there exist many other ensembles invariant w.r.t.\ the orthogonal transformations~\cite{Kostlan93,Kostlan02}). In particular, the induced distribution on the space of hypersurfaces on~$\mathbb{RP}^m$ is also invariant w.r.t.\ the unitary group, which justifies our description of the nodal set of the Kostlan ensemble as a natural model for a `typical' real projective hypersurface.

\vspace{0.2cm}
As mentioned above, it will be convenient to consider $f_{n}$ as a Gaussian random field on the unit sphere $\mathbb{S}^m$, and henceforth we take exclusively this view. Computing
explicitly from~\eqref{eq:fn Kostlan def}, one may evaluate its covariance kernel~$\kappa_{n}:\mathbb{S}^{m}\times\mathbb{S}^{m}\rightarrow\R$ to be
\begin{equation}
\label{eq:covar kn Kostlan}
\kappa_{n}(x,y)=\E[f_{n}(x)\cdot f_{n}(y)] = \left(\langle x,y\rangle\right)^{n} = (\cos{\theta(x,y)})^{n},
\end{equation}
where for $x,y\in\mathbb{S}^{m}$ we denote $\theta(x,y)$ to be the angle between $x$ and $y$, also equal to the spherical distance between these points; this covariance kernel determines $f_{n}$ uniquely via Kolmogorov's Theorem.

The random field $f_{n}$ on $\mathbb{S}^{m}$ is of high merit since it is rotationally invariant and also admits a natural scaling around every point understood in the following way.
Let us fix $x_0 \in \mathbb{S}^{m}$, and define the {\em scaled} covariance kernel on $\R^{m}\times \R^{m}$
\begin{equation}
\label{eq:Kxn->Kinf}
K_{x_0;n}(x,y)= \kappa_{n}\left(\exp_{x_0}\left(\frac{x}{\sqrt{n}}\right),\exp_{x_0}\left(\frac{y}{\sqrt{n}}\right)\right),
\end{equation}
where $\exp_{x_0}:\R^{m}\rightarrow \mathbb{S}^{m}$ is the exponential map on the sphere based at~$x_0$. Then, as is shown formally in section \ref{sec:prf thm Kostlan} below, the scaled covariance $K_{x_0;n}(x,y)$ satisfies the convergence
\begin{equation}
\label{eq:Kx(u,v)->Kinfty(u-v)}
K_{x_0;n}(x,y)\rightarrow K_{\infty}(x,y)=e^{-\|x-y\|^{2}/2}
\end{equation}
along with all its derivatives, locally uniformly in $x,y \in \R^{m}$; the r.h.s.\ of \eqref{eq:Kx(u,v)->Kinfty(u-v)} is the defining covariance kernel of the Bargmann-Fock field on $\mathbb{R}^m$, discussed further below.

\subsection{RSW estimates for random subsets of Euclidean space}
\label{s:perc}

In percolation theory the RSW estimates ~\cite{Ru,SeWe} are uniform lower bounds for {\em crossing probabilities} of various percolation processes, most fundamentally for Bernoulli percolation. These are a crucial input into establishing the more refined properties of percolation processes, such as the sharpness of the phase transition and scaling limits for the interfaces of percolation clusters.
\vspace{0.2cm}

Let $\Tc$ be a periodic lattice (i.e.\ a periodic set of nodes and edges/bonds between each pair of adjacent nodes), and $p\in [0,1]$ a number. In Bernoulli bond percolation each edge of $\Tc$ is independently either open with probability $p$ or closed with probability $1-p$. This defines a (random) percolation subgraph $\Gc$ of $\Tc$ containing all vertices and only open edges. Alternatively one can think of colouring edges independently black (with probability $p$) or white (with probability $1-p$). In this case $\Gc$ is the black sub-graph.

A rather simple argument shows that there exists a {\em critical probability}: a number $p_{c}\in (0,1)$ such that for all $p>p_{c}$ the graph $\Gc$ a.s.\ contains an infinite {\em percolation cluster} (connected component of $\Gc$), and for all $p<p_{c}$ a.s.\ no such component exists. The more subtle behaviour of the percolation process for $p=p_{c}$, {\em critical percolation}, is of high intrinsic interest. Apart from being one of the most studied lattice models, it is also believed ~\cite{BS02} to represent the nodal structure of Laplace eigenfunctions on `generic' chaotic manifolds, in the high energy limit. For $\Tc$ possessing sufficient symmetries, the corresponding critical probability should be equal $p_{c}=1/2$; for the square lattice this was established rigorously by Kesten ~\cite{Ke}.

Let us assume that the lattice $\Tc$ is regularly embedded in $\R^2$, e.g.\
the canonical embedding of the square lattice as $\mathbb{Z}^2$ in $\R^2$.
For $\rho>1$, $s>0$ and $x_{0}\in \R^{2}$ a {\em box-crossing event} is the event that a rectangle
\[
R = x_{0}+[-\rho s/2,\rho s/2]\times [-s/2,s/2]
\]
centred at $x_{0}$ of size $s\times \rho s$ is traversed horizontally by a black cluster, i.e.\ there exists a connected component $\Cc$ of $\Gc$ such that $\Cc$, restricted to $R$, intersects both $\{x_0 -\rho s/2\} \times [-s/2,s/2]$ and $\{x_0 + \rho s/2\} \times [-s/2,s/2]$.  The basic RSW estimates for critical percolation are the assertion that, for every $\rho>1$ the corresponding crossing probability is bounded away from $0$ uniformly in the scale $s>0$, i.e.\ there exists a number $c(\rho)>0$ such that the probability of a box-crossing event is $\ge c(\rho)$ for all $s>0$, $x_{0}\in\R^{2}$. The analogous estimates hold for {\em quads}, i.e.\ triples $Q = (D;\gamma,\gamma')$, where $D$ is a piecewise-smooth domain, and $\gamma,\gamma'\subseteq\partial U$ are two disjoint boundary curves; in this case the RSW estimates assert that there exists a constant $c(D; \gamma,\gamma')>0$ such that the probability $p(D;\gamma,\gamma';s) $ that $sD = \{sx: x \in D\}$ contains a black cluster intersecting both $s\gamma$ and $s\gamma'$ is at least $c(D; \gamma,\gamma')$ for every $s>0$.

\vspace{0.2cm}
In the more general setting of random subsets of Euclidean space, Tassion ~\cite{Tas16} recently showed the validity of RSW estimates for the {\em Voronoi percolation}. Let $\Pc\subseteq\R^{2}$ be a Poisson point process on $\R^{2}$ with unit intensity, and for each $x\in\Pc$ construct the associated (random) Voronoi cell
\[ \Cc_{x} = \{z\in \R^{2}:\: \forall y\in \Pc\setminus\{x\}\rightarrow d(z,y)\ge d(z,x)\}; \]
the various Voronoi cells tile the plane disjointly save for boundary overlaps. Each of the cells is coloured black or white independently with probabilities $p$ and $1-p$
respectively; here again, by a duality argument, the critical probability is $p_{c}=1/2$ \cite{BoRi06}. In this setting Tassion ~\cite{Tas16} proved that RSW estimates hold on all scales; a somewhat weaker version due to Bolobas-Riordan ~\cite{BoRi06} established that the RSW estimates hold for an unbounded subsequence of scales.

\subsection{RSW estimates for the Bargmann-Fock space}

\label{sec:RSW Beffara-Gayet}

Our starting point is the recent work of Beffara-Gayet~\cite{BG16} that established the RSW estimates for the {\em nodal sets} of a family of stationary smooth Gaussian random fields on $\mathbb{R}^2$,
with {\em positive} and {\em rapidly decaying} correlations satisfying {\em sufficient symmetry}; the motivating and main example of such a field was the scaling limit of the Kostlan ensemble \eqref{eq:fn Kostlan def} for
dimension $m=2$. To the best of our knowledge, along with the very recent announcement
of Nazarov-Sodin on the variance of the number of nodal domains (to be published), Beffara-Gayet's result is the only heretofore known rigorous evidence or manifestation for the conjectured connections ~\cite{BS02} between percolation theory and nodal patterns.

\vspace{2mm}
Let $g_{\infty}:\R^{2}\rightarrow\R$ be the random field indexed by $(x_{1},x_{2})\in\R^{2}$
corresponding to the covariance kernel $K_{\infty}$
on the r.h.s.\ of \eqref{eq:Kx(u,v)->Kinfty(u-v)}. Then $g_{\infty}$ is an isotropic random field, a.s.\ smooth, which may be constructed explicitly as the series
\begin{equation}
\label{eq:ginf series def}
g_{\infty}(x) = \sum\limits_{i,j=0}^{\infty}a_{ij}\frac{1}{\sqrt{i!j!}} x_{1}^{i}x_{2}^{j}
\end{equation}
with $\{a_{ij}\}$ i.i.d.\ standard Gaussian random variables, and where the convergence is understood locally uniformly; hence the sample paths of $g_{\infty}$ are a.s.\ {\em real analytic}.
Equivalently, recall the Bargmann-Fock space in section~\ref{s:kost} above, and define the space $\Fc$ of analytic functions on~$\R^{2}$ that admit an analytic extension to $\C^2$ which lies in the Bargmann-Fock space; equip this space with the scalar product $\langle \cdot, \cdot \rangle_{BF}$ induced from \eqref{e:bfs}. We may then think of $g_{\infty}$ in~\eqref{eq:ginf series def} as the canonical Gaussian element of $\Fc$ (c.f.\ ~\cite[Appendix A.1]{BG16}), normalised to have unit variance.

Define the nodal components $\{\Cc_{i}\}_{i}$ of $g_{\infty}$ to be the connected components of
the nodal set~$g_{\infty}^{-1}(0)$, and the nodal domains $\{\Dc_{i}\}_{i}$ of $g_{\infty}$ to be the connected components of the complement
$\R^{2}\setminus g_{\infty}^{-1}(0)$ of the nodal set; a.s.\ all the nodal components $\{\Cc_{i}\}$ are simple smooth curves. Nazarov and Sodin ~\cite{NS15} proved that the number of nodal components $\Cc_{i}$ entirely contained in the disk of radius $R$ is asymptotic to
$c_{NS}\cdot R^{2}$ with $c_{NS}>0$ the `Nazarov-Sodin constant of $g_{\infty}$'. The main result of Beffara-Gayet ~\cite[Theorem 1.1]{BG16} was that the RSW estimates hold for the complement of the nodal set on all scales, and for the nodal set itself on all sufficiently large scales. The restriction to sufficiently large scales is natural, since the probability that the nodal set intersects a domain tends to zero with the size of the domain.

\vspace{2mm}
As was mentioned at the beginning of section~\ref{sec:RSW Beffara-Gayet}, other than for $g_{\infty}$ the result in \cite{BG16} also applies to a (somewhat limited) family of Gaussian random fields; these are fields which have sufficiently nice properties so that Tassion's aforementioned techniques and ideas are applicable. Since Tassion's ideas are also instrumental for the proofs of the results of this paper, the generality of our results are also limited in a similar way.

\subsection{Statement of the principal result: RSW estimates for the Kostlan ensemble}

Our aim is to prove the analogous RSW estimates for the $m=2$-dimensional Kostlan ensemble~\eqref{eq:fn Kostlan def}, without passing to the limit. In light of the discussion in subsection \ref{s:kost} above, these estimates can be interpreted as uniform bounds on crossing probabilities for a `typical' algebraic curve on $\mathbb{RP}^2$. The RSW estimates that we establish are stronger than those which can be deduced from the corresponding estimates~\cite{BG16} for the Bargmann-Fock limit field \eqref{eq:ginf series def}, since they also hold on macroscopic scales. Indeed, our main result (Theorem~\ref{t:kos} below) establishes RSW estimates that hold uniformly on the projective space (or sphere, after removal of antipodal points).

\vspace{2mm}
Naturally one could try to work in the same Euclidean setting as was used to establish the RSW estimates~\cite{BG16} on the Bargmann-Fock limit field \eqref{eq:ginf series def}. One would then consider ~\cite[p. 6]{BG16} the projection of $f_{n}$ on the Euclidean space via the natural embedding $\pi:\R^{2}\hookrightarrow \RP^{2}$ with $x=(x_{1},x_{2})\mapsto (1:x_{1}:x_{2})$; in this case the corresponding covariance kernel of
\[ \tilde{f_{n}}(x) = f_{n}(\pi(x)) \]
on $\R^{2}$, normalised to be unit variance, is
\begin{equation}
\label{eq:covar lamn Kostlan Euclid}
\lambda_{n}(x,y)= \frac{(1+\langle x,y \rangle)^{n}}{(1+\|x\|^{2})^{n/2} \cdot (1+\|y\|^{2})^{n/2}},
\end{equation}
where $\|\cdot\|$ is the standard Euclidean norm on $\R^{2}$.

Unfortunately this model does not enjoy particularly nice properties, being neither stationarity nor invariant w.r.t.\ negation of the second coordinate, key ingredients in Beffara-Gayet's (and Tassion's) argument. Our primary observation is that these properties \textit{do} hold for the spherical ensemble \eqref{eq:fn Kostlan def}. This will allow us to establish the RSW estimates directly for the spherical model, our main result, which we now prepare the ground for.

\vspace{2mm}

We begin by formally defining the RSW estimates as they apply to general sequences of random sets on the sphere; later this will be extended in an analogous way to the flat torus, see section~\ref{sec:RSW est symmetric} below. Let us start by introducing `quads' and their associated crossing events (c.f.\ the discussion in section \ref{s:perc} above).

\begin{definition}[Quads and crossing events]
\label{d:quad}
A quad $Q = (D; \gamma, \gamma')$ is a piecewise-smooth simply-connected (spherical) domain $D \subset \mathbb{S}^2$ and the choice of two disjoint boundary arcs  $\gamma,\gamma' \subset \partial D$. When we consider a quad $Q$ as a set, we will identify it with the closure of~$D$. For each $X \subseteq \mathbb{S}^2$ we denote by $\rm{Quad}_X$ the collection of quads $Q \subseteq X$.

To each quad $Q = (D; \gamma, \gamma')$ and random subset $\mathcal{S}$ of $\mathbb{S}^2$ we associate the `crossing event' $\Cc_Q(\Sc)$ that a connected component of $\Sc$, restricted to $D$, intersects both $\gamma$ and $\gamma'$. We shall sometimes use a phrase such as `$Q$ is crossed by $\Sc$' to describe the event $\Cc_Q(\Sc)$.
\end{definition}

Rather than stating the RSW estimates for \textit{rescaled} boxes or quads (as was done in \cite{BG16} and \cite{Tas16}  for instance), in non-Euclidean settings it is natural to state these estimates for a more general class of quads that can be `uniformly crossed by chains of boxes'; we introduce this concept now as it applies to the sphere. The following definition is rather technical but it is well illustrated by Figure \ref{fig:box chain}.

\begin{figure}
\centering
\includegraphics[width=0.5\textwidth]{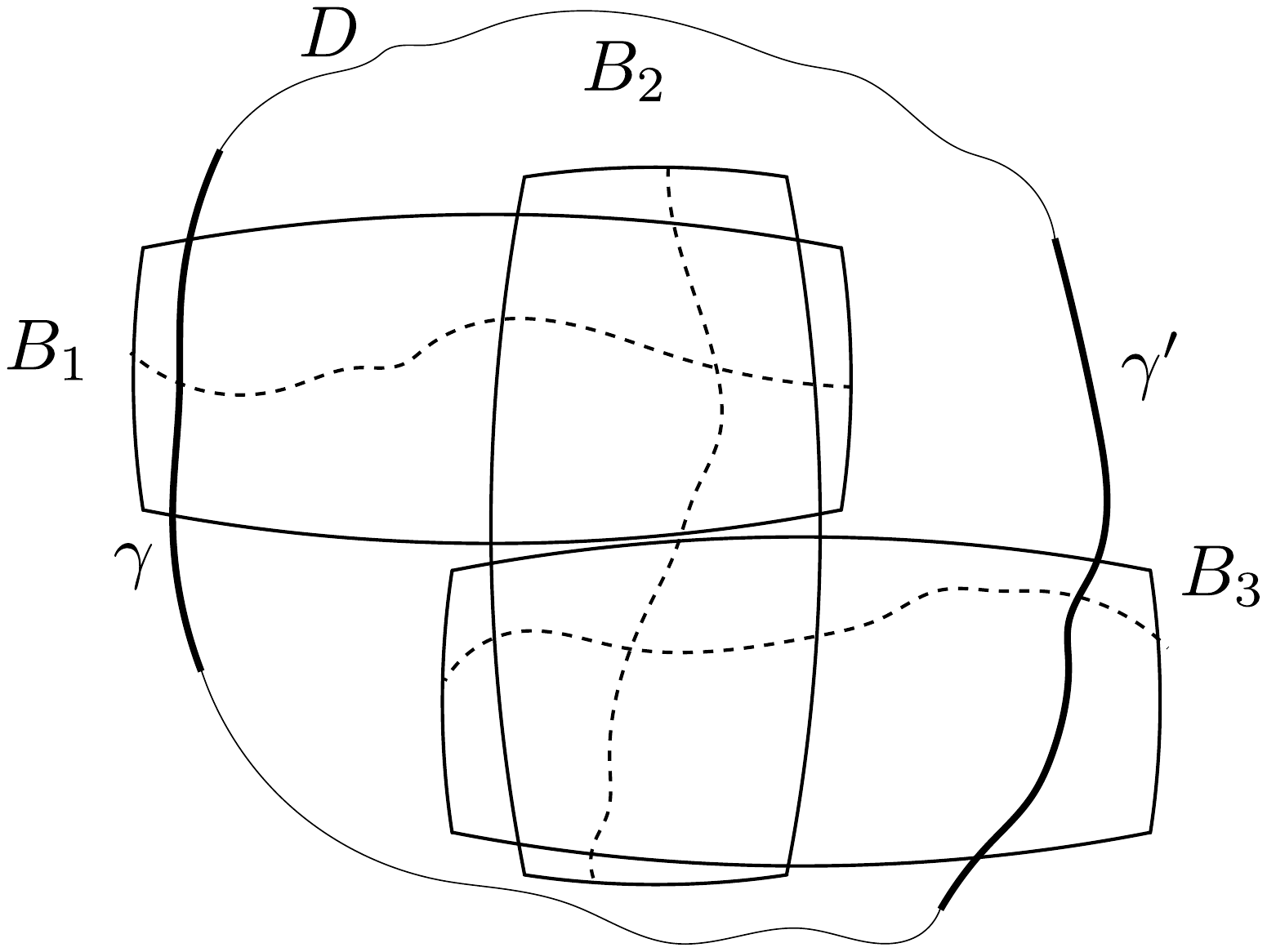}
\caption{A box-chain $(B_i)$ of length three that crosses a quad $Q =(D;\gamma,\gamma')$; if each of the boxes $B_i$ are crossed by a set $\Sc$ then so is the quad $Q$.}
\label{fig:box chain}
\end{figure}

\begin{definition}[Spherical boxes and box-chains]
\label{d:box chain}
~
\begin{enumerate}

\item For each $a, b > 0$, an $a \times b$ (spherical) rectangle $D \subset \mathbb{S}^2$ is a simply-connected domain that is bounded by four geodesic line-segments, with all four internal angles equal, and such that the non-adjacent pairs of boundary components have length $a$ and~$b$ respectively.  We refer to the four boundary components of a rectangle as its `sides', and shall call a rectangle with equal side-lengths a `square'.

\item An $a \times b$ box $B$ is a quad $Q = (D; \gamma, \gamma')$ in $\mathbb{S}^2$ such that $D$ is an $a \times b$ rectangle and such that $\gamma$ and $\gamma'$ are the opposite sides of length $a$. We refer to the sides of $B$ other than $\gamma$ and $\gamma'$ as the `lateral' sides. For each $X \subseteq \mathbb{S}^2$, $c \ge 1$ and $s > 0$, we denote by $\rm{Box}_{X; c }(s)$ the collection of all $a \times b$ boxes $B \subseteq X$ such that $s \le a , b \le cs$.

\item
A curve $\eta \subset D$ is said to `transversally cross' a box $B = (D; \gamma, \gamma')$ if a connected component of $\eta$, restricted to $D$, intersects both of the lateral sides of $D$; in particular $\gamma$ and $\gamma'$ always transversally cross $B$.

\item
A box $B = (D; \gamma, \gamma')$ is said to `transversally cross' another box $\hat{B}$ if both of the lateral sides of $D$ transversally cross $\hat{B}$; this definition is symmetric in the sense that it also implies that $\hat{B}$ transversally crosses $B$.

\item
A `box-chain' of length $n$ is a finite set $\{B_i\}_{1 \le i \le n}$ of boxes such that, for each $i = 2, \ldots, n$, $B_i$
transversally crosses $B_{i-1}$. A quad $Q = (D; \gamma, \gamma')$ is said to be `crossed' by a box-chain $\{B_i\}_{1 \le i \le n}$ if $\gamma$ transversally crosses $B_1$, $\gamma'$ transversally crosses $B_n$, and $\cup_{2 \le i \le n-1} B_i \subseteq D$.

\end{enumerate}

\end{definition}

The relevance of box-chains to RSW estimates can be seen from the following. Let $Q$ be a quad that is crossed by a box-chain $\{B_i\}$, and let $\Sc$ be a random subset of $\mathbb{S}^2$. Then if the event $\mathcal{C}_{B_i}(\Sc)$ holds for each $i$, so does the event $\Cc_Q(\Sc)$ (see Figure \ref{fig:box chain}). In other words, one may bound the probabilities of crossings of quads by controlling the crossings of box-chains instead. This motivates the following definition.

\begin{definition}[Quads that are uniformly crossed by box-chains]
\label{d:unifcross}
For each $X \subseteq \mathbb{S}^2$, $c \ge 1$ and $s > 0$, we denote by $\rm{Unif}_{X;c}(s)$ the collection of all quads $Q \in \rm{Quad}_X$ that are crossed by a box-chain $\{B_i\}_{1 \le i \le n}$ of length $n \le c$ such that $B_i \in \rm{Box}_{X;c}(s)$ for each $i$.
\end{definition}

The property of quads being uniformly crossed by box-chains generalises the notion of scale invariance on the sphere, with the parameter $c$ in the definition of $\rm{Unif}_{X;c}(s)$ playing the role of the `aspect ratio'. One can check, for instance, that for each quad $Q = (D; \gamma, \gamma')$ there is a $c > 1$ such that $\rm{Unif}_{\mathbb{S}^2;c}(s/c)$ contains the rescaled quad $sQ = (sD; s\gamma, s\gamma')$ for each $s \in (0, 1]$, where $sA$ denotes linear rescaling of the set $A$ along the unique geodesic to the origin (deleting the antipodal point if necessary). This can be seen by observing that, although rescaling does not preserve geodesics on the sphere, the resulting distortion is uniformly controlled on all small enough scales.

The property of being uniformly crossed by box-chains is also closely related to conformal invariants. One can check, for instance, that if a quad  $Q = (D; \gamma, \gamma')$ is crossed by a box-chain of length $n$ consisting of boxes from $\rm{Box}_{X;c}(s)$, then the extremal distance from $\gamma$ to $\gamma'$ in $D$ (which is the only conformal invariant of $Q$) is bounded above by $cn$, independently of $s$. In particular, for $Q \in \rm{Unif}_{X;c}(s)$ the extremal distance is uniformly bounded above by $c^2$.

\vspace{0.2cm}
We next introduce the RSW estimates as they apply to the sphere; these give a uniform lower bound on crossing probability for quads that are uniformly crossed by box-chains. We state the RSW estimates for arbitrary sequences of random subsets.

\begin{definition}[RSW estimates]
\label{d:rsw}
Let~$(\mathcal{S}_n)_{n \in \mathbb{N}}$ be a sequence of random subsets of~$\mathbb{S}^2$, let $X \subseteq \mathbb{S}^2$, and let $s_n \ge 0$ be a sequence satisfying $s_n \to 0$ as $n \to \infty$. We say that the sequence $(\mathcal{S}_n)_{n \in \mathbb{N}}$ `satisfies the RSW estimates on $X$ down to the scale $s_n$' if for every $c > 1$ there exists a $C > 0$ such that
\begin{equation}
\label{e:rsw1}
\liminf_{n \to \infty} \, \inf_{ s > C s_n } \, \inf_{ Q \in \rm{Unif}_{X;c}(s)}  \,  \mathbb{P}(\mathcal{C}_Q(\mathcal{S}_n)) > 0 .
\end{equation}
We say that the sequence $(\mathcal{S}_n)_{n \in \mathbb{N}}$ `satisfies the RSW estimates on $X$ on all scales' if \eqref{e:rsw1} holds for $s_n \equiv 0$. \end{definition}

Strictly speaking we should restrict the definition of the RSW estimates in \eqref{e:rsw1} to only hold for quads $Q$ such that $\Cc_Q(\Sc_n)$ is measurable. However, since we work only with $\Sc_n$ being level sets or excursion sets of a.s.\ $C^2$ Gaussian random fields, the events $\Cc_Q(\Sc_n)$ are always measurable and so we will ignore this technicality.

\vspace{2mm}

We are now ready to state our main result. Recall that the \textit{nodal sets} of the Kostlan ensemble are the (random) subsets $\mathcal{N}_n = f^{-1}(0)$ of the sphere; we also consider their complements, $\mathbb{S}^2 \setminus \mathcal{N}_n$.
Our principal result asserts that the RSW estimates in Definition \ref{d:rsw} hold down to the scale $n^{-1/2}$ for the nodal sets of the Kostlan ensemble, and on all scales for their complements (the latter estimates give a lower bound for the probability of a domain being crossed by a single nodal domain).

\begin{theorem}[RSW estimates for the nodal sets of the Kostlan ensemble]
\label{t:kos}
Let $X \subset \mathbb{S}^2$ be a subset whose closure does not contain pairs of antipodal points, and let $s_n = n^{-1/2}$. Then the following hold:
\begin{enumerate}
\item
The nodal sets of the Kostlan ensemble \eqref{eq:fn Kostlan def} on $\mathbb{S}^2$ satisfy the RSW estimates on~$X$ down to the scale $s_n$.

\item The complements of the nodal sets of the Kostlan ensemble \eqref{eq:fn Kostlan def} on $\mathbb{S}^2$ satisfy the RSW estimates on $X$ on all scales.
\end{enumerate}
\end{theorem}

We constrain the RSW estimates to apply only to a set $X$ whose closure does not contain pairs of antipodal points since the Kostlan ensemble is naturally defined on the projective space; indeed, the RSW estimates do not hold on the whole of the sphere, as certain crossing events on the sphere are impossible due to the identification of points on the projective space.

The scales on which we prove the RSW estimates in Theorem \ref{t:kos} are optimal in the sense that these estimates fail for the nodal set on smaller scales than $s_n = n^{-1/2}$. To see this, recall that $s_n$ is the scale on which the local uniform convergence of the ensemble in \eqref{eq:Kx(u,v)->Kinfty(u-v)} takes place (in what follows, we often refer to this as the `microscopic scale'), and since the probability that a nodal set crosses a quad in the limit field tends to zero as the size of the quad tends to zero, the same is true for the Kostlan ensemble on scales smaller than $s_n$.

\subsection{Acknowledgements}

The research leading to these results has received funding from the Engineering \& Physical Sciences Research Council (EPSRC) Fellowship EP/M002896/1 held by Dmitry Beliaev (D.B. \& S.M.), the  EPSRC Grant EP/N009436/1 held by Yan Fyodorov (S.M.), and the European Research Council under the European Union's Seventh Framework Programme (FP7/2007-2013), ERC grant agreement n$^{\text{o}}$ 335141 (I.W.) The authors would like to thank Damien Gayet, Mikhail Sodin and Dmitri Panov for useful discussions.

\smallskip
\section{Outline of the paper}

Theorem \ref{t:kos} is a particular case of a more general result, Theorem \ref{t:main} below, which asserts the RSW estimates for the nodal sets, and their complements, of general sequences of centred Gaussian random fields defined on smooth compact Riemannian manifolds $\Xb$ satisfying sufficient symmetries. In turn, the proof of Theorem \ref{t:main} has, as its main ingredient, the even more general Theorem \ref{t:rsw}, which asserts RSW estimates for abstract sequences of random sets obeying a natural scaling (as well as certain other conditions).
We believe that both the general Theorem~\ref{t:main} and the abstract Theorem \ref{t:rsw} are of independent interest.

\subsection{General RSW estimates for nodal sets of sequences of Gaussian random fields}
\label{sec:RSW est symmetric}

Let $\mathbb{X}$ be either the flat torus $\mathbb{T}^2 = \mathbb{R}^2 \setminus \mathbb{Z}^2$ or the unit sphere~$\mathbb{S}^2$, and equip $\mathbb{X}$ with a marked origin $0 \in \mathbb{X}$ and its natural metric $d(\cdot, \cdot)$. We consider a sequence $(f_{n})_{n \in \mathbb{N}}$ of Gaussian random fields defined on $\Xb$.

\vspace{0.2cm}
Our first task is to define the relevant RSW estimates, which will be the natural generalisation of Definition \ref{d:rsw} to $\mathbb{X}$. To begin we can define `quads' and `crossing events' analogously to Definition~\ref{d:quad} (i.e.\ replacing $\mathbb{S}^2$ everywhere with $\mathbb{X}$). Before discussing `boxes' as in Definition~\ref{d:box chain}, we need to alter slightly the definition of `rectangles' in the case $\mathbb{X} = \mathbb{T}^2$, namely restricting their sides to be parallel to the axes; this is so that we may work with fields on~$\mathbb{T}^2$ that are not assumed to be rotationally symmetric. For clarity, we restate this definition (with the difference to Definition \ref{d:box chain} emphasised).

\begin{definition}[Toral rectangles]
\label{d:box2}
For each $a, b > 0$, an $a \times b$ (toral) rectangle $D \subset \mathbb{T}^2$ is a simply-connected domain that is bounded by four geodesic line-segments \textbf{that are parallel to the axes}, with all four internal angles equal, and such that the non-adjacent pairs of boundary components have length $a$ and $b$ respectively.
\end{definition}

With this definition of toral rectangles, `boxes' are defined as in Definition \ref{d:box chain}; the notion of `box-crossings' of quads, as well as the set $\rm{Unif}_{X;c}(s)$, are then analogous to in Definitions~\ref{d:box chain} and~\ref{d:unifcross}. Finally, RSW estimates are defined analogously to Definition~\ref{d:rsw}.

\vspace{0.2cm}
We next state various conditions that we impose on the Gaussian random fields $(f_n)_{n \in \mathbb{N}}$ that we consider; these are most conveniently framed in terms of their covariance kernels. We first describe a set of relevant symmetries that these covariance kernels must satisfy. These symmetries naturally limit the choice of the underlying space to $\mathbb{T}^2$ and $\mathbb{S}^2$.

\begin{definition}[Symmetry]
\label{a:symmetry}
We say that a covariance kernel on $\mathbb{X}$ is `symmetric' if:
\begin{enumerate}
\item In the case $\mathbb{X} = \mathbb{S}^2$, it is rotationally invariant and symmetric w.r.t.\ reflection in any great circle;
\item In the case $\mathbb{X} = \mathbb{T}^2$, it is stationary and possesses the $D_4$ symmetry, i.e., it is invariant w.r.t.\ horizontal reflection and rotation by $\pi/2$.
\end{enumerate}
\end{definition}

Next we impose certain smoothness and non-degeneracy conditions on symmetric covariance kernels
(in the sense of Definition \ref{a:symmetry}). When we work with symmetric covariance kernels, we often naturally consider them as functions of one variable, i.e.\ setting $\kappa(x) = \kappa(0, x)$ (with a slight abuse of notation).

\begin{assumption}[Smoothness and non-degeneracy]
\label{a:nondegen}
A symmetric covariance kernel $\kappa$ on~$\mathbb{X}$ satisfies the following:
\begin{enumerate}
\item The function $\kappa(x)$ is $C^6$;

\item The Hessian $H_{\kappa}(0)$ of $\kappa$ at the origin is positive-definite.

\end{enumerate}
\end{assumption}

By the standard theory \cite{Adler-Taylor}, Assumption \ref{a:nondegen} guarantees that the associated random field
is a.s.\ $C^{2}$, and its nodal set a.s.\ consists of $C^2$ curves diffeomorphic to circles.

\vspace{0.2cm}
Finally, we define the concept of `local uniform convergence' for a sequence of covariance kernels on $\mathbb{X}$, generalising our discussion of the local limit \eqref{eq:Kx(u,v)->Kinfty(u-v)} of the Kostlan ensemble above. Let $\Phi : \mathbb{R}^2 \to \mathbb{X}$ denote a smooth map that is locally a linear isometry (i.e.\ such that $\Phi(0) = 0$ and the differential $d\Phi$ is a linear isometry); in the case $\mathbb{X} = \mathbb{T}^2 $ one may take the covering map for instance, whereas in the case $\mathbb{X} = \mathbb{S}^2$ one may take the exponential map based at the origin, as in \eqref{eq:Kxn->Kinf}.

\begin{definition}[Local uniform convergence of the covariance kernels near the origin; c.f. ~\cite{NS15}, Definition $2$]
\label{def:loc unif scal}
For a sequence $s_n > 0$ satisfying $s_n \to 0$ as $n \to \infty$ we say that covariance kernels $(\kappa_n)_{n \in \mathbb{N}}$ on $\mathbb{X}$ `converge locally uniformly near the origin on the scale $s_n$' if there exists a symmetric covariance kernel $K_\infty$ on $\mathbb{R}^2$, satisfying Assumption \ref{a:nondegen}, and an open set $U \subseteq \mathbb{R}^2$ containing the origin such that, as $n \to \infty$, for $x, y \in U$ uniformly,
\begin{equation}
\label{eq:Kn(x,y)->Kinf(x-y)}
K_n(x, y) = \kappa_n(  \Phi(s_n x),  \Phi( s_n y) ) \to   K_\infty(x-y) .
\end{equation}
We say that the covariance kernels $(\kappa_n)_{n \in \mathbb{N}}$ on $\mathbb{X}$ `converge locally uniformly near the origin on the scale $s_n$ along with their first four derivatives' if the above holds also for all partial derivatives $K_n$ of order up to $4$.
\end{definition}

We are now ready to state our general result Theorem \ref{t:main}; the proof that Theorem \ref{t:kos} is a special case
of Theorem \ref{t:main} is given in section \ref{sec:prf thm Kostlan} below.

\begin{theorem}[RSW estimates for general sequences of Gaussian random fields]
\label{t:main}
Let $(f_n)_{n \in \mathbb{N}}$ be a sequence of centred Gaussian random fields on $\mathbb{X}$ with respective covariance kernels $\kappa_n$. Suppose that there exists a constant $\eta > 0$, a set $X \subseteq \mathbb{X}$, and a  sequence $s_n >0$ satisfying $s_n \to 0$ as $n \to \infty$, such that the following hold:
\begin{enumerate}
\item Symmetry: The covariance kernels $\kappa_n$ are symmetric in the sense of Definition \ref{a:symmetry}.
\item Smoothness and non-degeneracy: The covariance kernels $\kappa_n$ satisfy Assumption \ref{a:nondegen}.
\item Local uniform convergence near the origin: The covariance kernels $\kappa_n$ converge locally uniformly near the
origin on the scale $s_n$ along with their first four derivatives.
\item Asymptotically non-negative correlations:
\begin{equation}
\label{eq:asymp pos}
\lim\limits_{n \to \infty}    s_n^{-12 - \eta}  \sup_{x,y \in X }  (\kappa_n( x, y )  \wedge 0 )  = 0 .
\end{equation}
\item Uniform rapid decay of correlations:
\begin{equation}
\label{eq:rapid decay corr}
\lim_{C \to \infty} \,  \limsup_{n \to \infty}   \sup_{ \substack{x, y \in X, \\   d(x,y) > C s_n }} \left(d(x,y)s_{n}^{-1}\right)^{18 + \eta} |\kappa_n( x , y)| = 0 .
\end{equation}

\end{enumerate}

Then the nodal sets of $f_n$ satisfy the RSW estimates on $X$ down to the scale $s_n$, and the complements of the nodal sets of $f_n$ satisfy the RSW estimates on $X$ on all scales.

\end{theorem}

In Theorem \ref{t:main} the covariance kernels are, in principle, allowed to be negative, unlike for the Gaussian random field considered in \cite{BG16}; this is crucial for our application to the Kostlan ensemble since, for $n$ odd, the Kostlan ensemble is only positively correlated within a subset of the sphere. Nevertheless, since the negative correlations in the Kostlan ensemble decay exponentially rapidly as a function of $n$ for any subset $X$ whose closure does not contain antipodal points, condition \eqref{eq:rapid decay corr} is satisfied.

In regards to the nature of the exponents $12$ and $18$ in \eqref{eq:asymp pos} and \eqref{eq:rapid decay corr} respectively,
these are certainly not optimal for the claimed results, and are chosen mainly for simplicity. In fact, using the somewhat more sophisticated methods in \cite{BM}, if we additionally assume local uniform convergence of the first \textit{six} derivatives of the covariance kernel we could reduce these exponents to $8$ and $12$ respectively. Moreover, with an extra assumption that the covariance kernels $\kappa_n$ are smooth with derivatives decaying at least as rapidly as the kernel, and if the local convergence \eqref{eq:Kn(x,y)->Kinf(x-y)} of the covariance kernels holds together with \textit{all} derivatives, using the method in \cite{BM2} we could further reduce these exponents to $4$ and $6$. For simplicity, we do not implement these improvements here; on the other hand the question of the optimal exponents in \eqref{eq:asymp pos} and \eqref{eq:rapid decay corr}
is of considerable importance.

\vspace{0.2cm}
To complete section \ref{sec:RSW est symmetric},
we give an example of an application of Theorem \ref{t:main} to a sequence of Gaussian random fields defined on the flat torus; that this example falls under the scope of Theorem \ref{t:main} is established in section \ref{sec:prf thm Kostlan}.

\begin{example}
\label{ex:torus}
Let $(f_n)_{n \in \mathbb{N}}$ be the sequence of centred stationary Gaussian random fields on the torus $\mathbb{T}^2$ with respective covariance kernels
\[ \kappa_n( x, y) =  \big( \cos \left( 2\pi (x_1 - y_1) \right) \cdot \cos \left( 2 \pi (x_2 - y_2)  \right) \big)^n , \quad (x, y) = ((x_1, x_2), (y_1, y_2) ) \in \mathbb{T}^2 \times \mathbb{T}^2. \]
Let $X \subseteq \mathbb{T}^2$ be subset whose closure contains no distinct points $(x_1,y_1)$ and $(x_2, y_2)$ such that $2(x_1 - y_1)$ and $2(x_2 - y_2)$ are integers. Then the nodal sets of $f_n$ satisfy the RSW estimates on $X$ down to the scale $s_n = n^{-1/2}$, and the complements of the nodal sets of $f_n$ satisfy the RSW estimates on $X$ on all scales.
\end{example}

The restriction on $X$ is imposed, once again, since the nodal sets are naturally defined on a quotient space of $\mathbb{T}^2$, and indeed the RSW estimates fail on the whole space.

\subsection{Overview of the proof of Theorem \ref{t:main}}
Similar to ~\cite{BG16}, the overall structure of the proof of Theorem \ref{t:main} consists of three main steps:

\begin{enumerate}

\item First (see section \ref{s:rsw}) we adapt an argument borrowed
from \cite{Tas16} to establish general RSW estimates for abstract sequences of random sets on $\mathbb{X}$
satisfying certain key assumptions. These general estimates are stated as Theorem \ref{t:rsw} below.

\item Next, we develop a sufficiently robust perturbation analysis that allows us to
apply the abstract RSW estimates to the complement of the nodal sets (in fact, separately to the positive
and negative excursion sets $f_{n}^{-1}(0,\infty)$ and $f_{n}^{-1}(-\infty,0)$ respectively)
of the Gaussian random fields $f_n$ in the setting of Theorem \ref{t:main} (see section \ref{sec:pert analys}).
This perturbation analysis is used in two key place in the proof, namely in establishing (i) that~\eqref{eq:rapid decay corr} guarantees the `asymptotic independence' of crossing events in well-separated domains, and (ii) that negative correlations satisfying ~\eqref{eq:asymp pos} have a negligible effect on crossing probabilities.
\item Finally, we again apply the `asymptotic independence' of crossing events to infer the RSW estimates for
the nodal sets from the RSW estimates for the complements of the nodal sets; this follows from similar arguments to those presented in \cite{BG16} (see the second
part of the proof of Theorem \ref{t:main} in section \ref{sec:main thm concl}).

\end{enumerate}

Despite the structural similarities between our approach to \cite{BG16}, we record three significant modifications that we make here. First, it is necessary to adapt the argument in \cite{Tas16} to handle the differences in our setting, namely: (i) the presence of a sequence of random sets rather than just a single random set; (ii) the fact that we work on (bounded) manifolds rather than the Euclidean plane; and (iii) in the spherical case, the positive curvature of the sphere. We believe these modifications to be of independent interest, since, to the best of our knowledge, no theory of RSW estimates exists outside the scope of Euclidean space, and our approach is the first step in this direction.

Second, we apply the general argument in \cite{Tas16} in a different manner compared to \cite{BG16}, in particular with regards to the treatment of the asymptotic independence of crossing events (see the comments at the end of section \ref{s:rsw2}). We believe that our approach yields a significant simplification of the argument presented in \cite{BG16}.
Finally, our argument is able to handle negative correlations, as long as these are asymptotically negligible; negative correlations were absent from the model considered in \cite{BG16}.

\subsection{RSW estimates for abstract sequences of random sets}
\label{s:rsw2}

We give here the statement of the abstract RSW estimates for general sequences of random sets on $\mathbb{X}$; establishing this abstract result is the first step towards the proof of Theorem \ref{t:main}. To this end, we first need to define the analogues of Euclidean annuli and their related crossing events.
Recall the definition of a `square' (definitions \ref{d:box chain} and \ref{d:box2}), and observe that a square has a natural `centre', being the unique interior point equidistant from each side.

\begin{definition}[Annuli and circuit crossing events]
\label{d:ann}~

\begin{enumerate}

\item For $b > a > 0$, an $a \times b$ `annulus' is a domain bounded between concentric squares with side-lengths $a$ and $b$ that
are `parallel', i.e.\ such that there is a single geodesic that intersects both boundary squares at the mid-points of opposite sides.

\item For each $X \subseteq \mathbb{X}$, $c \ge 1$, $r \ge 1$ and $s > 0$, we denote by $\rm{Ann}_{X;c;r}(s)$ the collection of
all $a \times b$ annuli $A \subset X$ such that $s \le a \le b \le cs$ and $b/a = r$.

\item To each annulus $A$ and random subset $\Sc$ of $\mathbb{X}$ we associate the `crossing event' $\Cc_A(\Sc)$ that a
connected component of $\Sc$, restricted to $A$, contains a `circuit' around $A$.

\end{enumerate}
\end{definition}

For each $r > 0$ and $v \in \mathbb{S}^1$, let $B(r) \subseteq \mathbb{X}$ denote the centred open ball of radius $r$, and let~$\mathcal{L}_v(r)$ denote the geodesic line-segment of length $r$, based at the origin, in direction $v$.
Our abstract RSW estimates are the following.

\begin{theorem}[RSW estimates for abstract sequences of random sets]
\label{t:rsw}
Let $(\mathcal{S}_n)_{n \in \mathbb{N}}$ be a collection of random sets on $\mathbb{X}$. Suppose that there exists a set $X \subseteq \mathbb{X}$ and a sequence $s_n > 0$ satisfying $s_n \to 0$ such that the following hold:
\begin{enumerate}

\item Non-degeneracy: For every $n \in \mathbb{N}$, $\mathbb{P}( 0 \in \partial S_n ) = 0$. Moreover, for every $v \in \mathbb{S}^1$,
\begin{equation}
\label{e:nondegenrsw}
\lim_{r \to 0} \, \liminf_{n \to \infty} \, \mathbb{P}(   \mathcal{L}_v(r s_n)  \cap \partial S  = \varnothing  ) = 1 .
 \end{equation}

\item Symmetry: For every $n \in \mathbb{N}$, the law of $\mathcal{S}_n$ satisfies the following symmetries: In the case $\mathbb{X} = \mathbb{S}^2$, invariance w.r.t.\
rotations and reflections w.r.t.\ great circles; in the case $\mathbb{X} = \mathbb{T}^2$, invariance w.r.t.\ translations, horizontal reflections and rotation by $\pi/2$.

\item Positive associations: For every $n \in \mathbb{N}$, all events measurable on $X$ and increasing w.r.t.\ the indicator function of $\mathcal{S}_n$ are positively correlated.

\item Crossing of square boxes on arbitrary scales:
\[
\liminf_{n \to \infty} \, \inf_{s >0 } \, \inf_{ B \in \rm{Box}_{X; 1}(s) } \, \mathbb{P}( \mathcal{C}_B(\mathcal{S}_n) ) > 0.
\]

\item Arbitrary crossings on the microscopic scale: There exists a number $\delta > 0$ such that
\[
\liminf_{n \to \infty} \, \inf_{ Q \in \rm{Quad}_{B(\delta s_n)  } } \, \mathbb{P} ( \mathcal{C}_Q(\mathcal{S}_n) ) > 0 . \]

\item Annular crossings of a `thick' annulus with high probability:
For each $c > 0$, $r > 1$ and $\varepsilon \in (0, 1)$, there exist $C_1, C_2 > 1$ such that, for all sufficiently large $n \in \mathbb{N}$ and all $s > C_1 s_n$, if
\[   \inf_{A \in \rm{Ann}_{X; C_2; r}(s)}  \P(\Cc_A(\Sc_n)  ) > c  , \]
then, for any $s \times C_2s$ annulus $A \subseteq X$,
\[ \P( \Cc_A(S_n))>1-\varepsilon. \]
\end{enumerate}

Then the collection of sets $(\mathcal{S}_n)_{n \in \mathbb{N}}$ satisfies the RSW estimates on $X$ on all scales.

\end{theorem}

Compared to the original setting in \cite{Tas16}, and also its application in \cite{BG16}, we have made two important modifications in the formulation of Theorem \ref{t:rsw}. First, since we are dealing with a {\em sequence} of random sets rather than a single random set, the conditions are all stated in a way that guarantees {\em uniform} control over all necessary quantities.

Second, we have formulated a general condition guaranteeing annular crossings with high probability (see condition (6)), rather than a working
under the more constraining assumption that the random sets in disjoint domains are asymptotically independent (as was done in \cite{Tas16} and \cite{BG16} for instance). This reformulation is useful because we want to work directly with the random fields defined on $\Xb$, rather than applying the general theorem to the \textit{discretised} version of the model, as was the approach in \cite{BG16}. We believe that this constitutes a significant simplification to the method, and could also be used to simplify the argument in \cite{BG16}.

\subsection{Summary of the remaining part of the paper}

The remainder of the paper is structured as follows. In section \ref{s:proof} we develop the perturbation analysis that is the crucial ingredient in applying the abstract Theorem \ref{s:rsw} to the setting of Gaussian random fields. We then combine this analysis with Theorem~\ref{t:rsw} to complete the proof of Theorem \ref{t:main}. We conclude the section by showing that Theorem \ref{t:kos} and Example \ref{ex:torus} fall within the scope of~Theorem \ref{t:main}.

In section \ref{s:rsw} we give the proof of the abstract Theorem~\ref{t:rsw}. This is similar to the argument in \cite{Tas16}, but with suitable modifications to adapt to our setting. Finally, in section \ref{s:pert} we complete the proof of the auxiliary results used in the perturbation analysis developed in section \ref{s:proof}.

\smallskip
\section{Proof of Theorem \ref{t:kos}: RSW estimates for Kostlan ensemble}
\label{s:proof}

In this section we complete the proof of Theorem \ref{t:main}, which implies Theorem \ref{t:kos} as a special case. The main ingredient will be a perturbation analysis that allows us to apply the abstract RSW estimates in Theorem \ref{t:rsw} to the positive (resp.\ negative) excursion sets of the Gaussian random fields in our setting; these have similarities to the methods in \cite{BG16} and \cite{NS15}.

The set-up for the perturbation analysis is the following. Let $(f_n)_{n \in  \mathbb{N}}$ be a collection of centred  Gaussian random fields on $\mathbb{X}$ whose respective covariance kernels $\kappa_n$ are symmetric in the sense of Definition \ref{a:symmetry} and satisfy Assumption \ref{a:nondegen}. Let
\begin{equation}
\label{eq:Sn excursion def}
\mathcal{S}_n^+=\{x\in\Xb: f_n(x)>0\}  \quad \text{and} \quad \mathcal{S}_n^-=\{x\in\Xb: f_n(x) < 0\}
\end{equation}
denote the positive and negative excursion sets of $f_n$ respectively. Without loss of generality we may assume that $f_n$ are unit variance, since a normalisation does not affect $\mathcal{S}^+_n$ or $\mathcal{S}^-_n$. We also assume that there exists a sequence $s_n >0$ satisfying $s_n \to 0$ as $n \to \infty$ such that the covariance kernels $\kappa_n$ converge locally uniformly (in the sense of Definition \ref{def:loc unif scal}) near the origin on the scale $s_n$ along with their first four derivative; let $K_{\infty}$ be the limiting covariance kernel. Let $\delta_0 > 0$ be sufficiently small that this uniform convergence holds on the ball $B(\delta_0)$.

\subsection{Perturbation analysis}
\label{sec:pert analys}

Our perturbation analysis proceeds in two steps. First we argue that, outside an event of a small probability, crossing events for the positive excursion set are determined by the signs of a Gaussian random field on a (deterministic) set of points of finite cardinality. Second, we control the effect of perturbations of the field on the probability of crossing events by controlling their impact on the finite-dimensional law associated to the signs of the
random field on the finitely many points described above (which, up to an event of a small probability, determine the crossing probabilities).

\vspace{2mm}

To state the main propositions of the perturbation analysis, we shall need to define an analogue of Euclidean `polygons' for the manifold $\mathbb{X}$.

\begin{definition}
\label{d:poly}
A polygon is a quad whose boundary consists of a finite number of geodesic line-segments. Similarly to boxes, we refer to the boundary components as `sides', and their length as `side-lengths'. For each $X \subseteq \mathbb{X}$, $c > 0$ and $s > 0$, we denote by $\rm{Poly}_{X; c}(s)$ the collection of polygons in $X$ with at most $c$ sides and with sides-lengths at most $cs$.
\end{definition}

The main propositions of the perturbation analysis are the following.

\begin{proposition}[Crossing events are determined by the signs of finitely many points]
\label{p:meas}
For sufficiently large $n \in \mathbb{N}$ the following holds. Fix $c,r> 1$. Then there exists a constant
\[ c_1=c_1(c; r; K_\infty; \delta_0) > 0, \]
such that, for all error thresholds $\varepsilon \in (0,1)$, scales $s > 0$, and
\[ Q \in \rm{Poly}_{\mathbb{X};c}(s) \cup \rm{Ann}_{\mathbb{X};c;r}(s), \]
 there exists a finite set $\mathcal{P} = \Pc(Q;K_\infty;\delta_{0}) \subset Q$ of cardinality at most
\[ |\Pc| < c_1 \left( \varepsilon^{-2} (s/s_n)^{6} \vee 1 \right) , \]
such that, outside an event of probability less than $\varepsilon$, the crossing event $\mathcal{C}_Q(\mathcal{S}^+_n)$
is determined by the signs of $f_n$ restricted to $\mathcal{P}$.
\end{proposition}

\begin{lemma}[Effect of perturbation on the signs of Gaussian vectors]
\label{l:comp}
Fix $\eta > 0$. Let $X$ and $Y$ be centred Gaussian vectors of dimension $n$ with respective covariance matrices $\Sigma_X$ and~$\Sigma_Y$, and let $\mathbb{P}_X$ and $\mathbb{P}_Y$ denote their respective laws. Suppose that $X$ is normalised to have unit variance, and define
\[ \delta = \max_{i,j \le n} | (\Sigma_X)_{i,j} - (\Sigma_Y)_{i,j} |  .\]
Then there exists a constant $c > 0$, depending only on $\eta$, such that, for all events $A$ that are measurable in $\mathbb{P}_X$ and $\mathbb{P}_Y$ w.r.t\ the signs of $X$ and $Y$ respectively, then the following hold:
\begin{enumerate}
\item If the diagonal entries of $\Sigma_Y - \Sigma_X$ are non-negative, then
\[ | \mathbb{P}_X(A ) - \mathbb{P}_Y(A) | < c \left( n^{3+\eta} \delta \right)^{1/4} .  \]
\item If in addition $\Sigma_Y - \Sigma_X$ is positive-definite, then
\[ | \mathbb{P}_X(A ) - \mathbb{P}_Y(A) | < c \left( n^{2 +\eta} \delta \right)^{1/4} .  \]
\end{enumerate}
\end{lemma}

The first statement of Lemma \ref{l:comp} is an improved version of \cite[Theorem 4.3]{BG16} and \cite[Proposition C.1]{BM}, implementing an idea from \cite{NS11}.

We stress that in Proposition \ref{p:meas}, once $K_\infty$ and $\delta_{0}$ are prescribed, neither the constant~$c_1$ nor the set $\mathcal{P}$, whose existence is established in Proposition \ref{p:meas}, depend on any other properties of $\kappa_{n}$. Hence we may choose a set $\mathcal{P}$ that works simultaneously for two different sequences of fields whose covariance kernels converge locally uniformly to $K_\infty$ on $B(\delta_0)$; this fact will be crucial in section \ref{s:macro} below.

\vspace{0.2cm}

The proof of Proposition \ref{p:meas} and Lemma \ref{l:comp} are given in section \ref{s:pert}. We mention here that the proof of Proposition \ref{p:meas} proceeds by controlling the event that the nodal set intersects any of the edges of a certain \textit{graph} more than once (see Lemma \ref{l:tp}). We then argue that, outside this event, all crossing events are determined by the signs of the field restricted to the vertices of the graph. An analogous result for Gaussian random fields on~$\mathbb{R}^2$ was established in \cite{BG16, BM}.
We now outline the two key consequences of the perturbation analysis in our setting.

\subsubsection{Asymptotic independence of crossing events}

The first consequence is that crossing events in disjoint polygons or annuli are asymptotically independent in
the limit $n\rightarrow\infty$, as long as their respective polygons or annuli are sufficiently well-separated; this follows
in particular from the condition \eqref{eq:rapid decay corr} of Theorem \ref{t:main}.

\begin{proposition}
\label{p:qi}
Suppose that there exists $X \subseteq \mathbb{X}$ and $\eta > 0$ such that \eqref{eq:rapid decay corr} holds. Then for each $c , r, k> 1$ and $\varepsilon > 0$ there is a $C > 0$ such that the following hold for all sufficiently large $n \in \mathbb{N}$:
\begin{equation}
\label{eq:cross prob asymp ind}
 \sup_{ s > C s_n }  \, \sup_{ \substack{ X_1, X_2 \subset X , \\ d(X_1, X_2) > s} } \, \sup_{ \substack{P_1 \in \rm{Poly}_{X_1;c}(s) , \\ P_2 \in \rm{Poly}_{X_2;c}(s) } } \left| \mathbb{P} \left(\mathcal{C}_{P_1}(\mathcal{S}^+_n) \cap \mathcal{C}_{P_2}(\mathcal{S}^-_n) \right) - \mathbb{P}(\mathcal{C}_{P_1}(\mathcal{S}^+_n))\cdot  \mathbb{P}( \mathcal{C}_{P_2}(\mathcal{S}^-_n) )  \right| < \varepsilon,
\end{equation}
and, for each $1 \le j \le k$,
\begin{align}
\label{eq:cross prob asymp ind2}
 & \sup_{ s > C s_n }  \, \sup_{ \substack{ X_1, X_2 \subset X , \\ d(X_1, X_2) > s} } \, \sup_{ \substack{ \{A_i\}_{0 \le i \le j-1} \subset \rm{Ann}_{X_1;c;r}(s) , \\ A_{j} \in \rm{Ann}_{X_2;c;r}(s) } }  \\
 & \nonumber \phantom{aaaaaaaaaaaaa}  \left| \mathbb{P} \big( \cap_{0 \le i \le j} \mathcal{C}_{A_i}^c(\mathcal{S}^+_n)  \big) - \mathbb{P}(\cap_{0 \le i \le j-1} \mathcal{C}^c_{A_i}(\mathcal{S}^+_n))\cdot  \mathbb{P}( \mathcal{C}^c_{A_j}(\mathcal{S}^+_n) )  \right| < \varepsilon.
\end{align}
\end{proposition}

Observe that while \eqref{eq:cross prob asymp ind2} is stated for the positive excursion sets (and for the complements of the events $\mathcal{C}_{A_i}$), \eqref{eq:cross prob asymp ind} is formulated to control the asymptotic independence \textit{between} crossing events $\mathcal{C}_{P_i}$ for the positive and negative excursion sets. This difference is solely due to how we intend to apply these results, and does not reflect limitations in their generality.
Before giving a proof for Proposition \ref{p:qi}, let us state and prove a crucial corollary of \eqref{eq:cross prob asymp ind2}, namely that condition \eqref{eq:rapid decay corr} of Theorem \ref{t:main} implies the `thick' annular crossing condition (6) of Theorem~\ref{t:rsw}.

\begin{corollary}
\label{c:qi}
Suppose that there exist $X \subseteq \mathbb{X}$ and $\eta > 0$ such that \eqref{eq:rapid decay corr} holds. Then for each  $c > 0$, $r > 1$ and $\varepsilon \in (0, 1)$, there exists $C_1, C_2 > 1$ such that, for all sufficiently large $n \in \mathbb{N}$ and all $s > C_1 s_n$, if
\[   \inf_{A \in \rm{Ann}_{X; C_2; r}(s)}  \P(\Cc_A(\Sc_n^+)  ) >c  \]
then, for every $s \times C_2s$ annulus $A \subseteq X$,
\[ \P( \Cc_A(\Sc_n^+))>1-\varepsilon. \]
\end{corollary}

\begin{proof}[Proof of Corollary \ref{c:qi} assuming Proposition \ref{p:qi}]
The idea of the proof is straightforward. If we take a large number of concentric well-separated annuli, then crossing events in these annuli are almost independent and have the same lower bound. This implies a crossing in one of them with high probability, and hence a crossing in a `thick' annulus with high probability.

Fix $c > 0$, $r > 1$ and $\varepsilon \in (0, 1)$. Since establishing the corollary for a $r > 1$ implies the corollary holds for every smaller $\bar r \in (1, r)$, we can and will assume that $r \ge 2$. We work with the collection $(A_{a, b})_{a<b}$ of $a \times b$ annuli centred at the origin that are `parallel', i.e.\ such that there is a single geodesic that intersects all boundary squares at the mid-points of opposite sides. In particular, for each $s > 0$ we introduce the sequence of disjoint annuli $\{A^s_i\}_{i \ge 0}$ defined by $A^s_i = A_{r^{2i}  s , r^{2i+1} s}$. Since $r \ge 2$ it holds that $d(A_{i}^{s}, A_{j}^{s}) > s$ for all $i \neq j$.

Let $k$ be an integer to be determined later, and set $C_2$ larger than $r^{2k+1}$. Fix $s > 0$ and consider an $cs \times C_2s$ annulus $A \subseteq X$. By symmetry we may assume $A = A_{s, C_2s}$, and hence $A^s_i \in \rm{Ann}_{X; C_2; r}(s)$ for each $0 \le i \le k$, which by assumption implies that
\begin{equation}
\label{e:qi1}
 \P( \Cc_{A^s_i}(S_n^+))> c .
 \end{equation}

Now, since $d(X_i, X_j) > s$, an application of \eqref{eq:cross prob asymp ind2} in Proposition \ref{p:qi} yields a $C_1 > 0$ such that, for sufficiently large $n$ and all $s > C_1 s_n$ and $j = 0 , \ldots , k$,
\[   \left|\Pb \left(  \cap_{i = 1, \ldots, j}  \, \Cc_{A^s_i}^c (\Sc_{n}^{+})  \right)
-  \Pb(\cap_{i = 0, \ldots, j-1}  \Cc_{A^s_i}^c (\Sc_{n}^{+}) ) \cdot \Pb \big(\Cc_{A^s_j}^c (\Sc_{n}^{+}) \big)   \right|  < \varepsilon / (2c) . \]
Combined with \eqref{e:qi1} this implies that
\[   \Pb(  \Cc_{A^s_i  }(\Sc_n^+) \text{ does not occur for } i =  0, \ldots, k )  <  f_{c; \varepsilon}^k(1-c)  ,\]
where $f^k_{c;\varepsilon}(x)$ denotes the $k$-fold iteration of the map $x \mapsto (1-c)x + \varepsilon / (2c)$.
One may check that $f^k_{c;\varepsilon}(1-c) \to \varepsilon/2$ as $k \to \infty$, and hence we may choose a $k$ sufficiently large such that
\[   \Pb(  \Cc_{A^s_i  }(\Sc_n^+) \text{ does not occur for } i =  0, \ldots, k )  < \varepsilon .\]
Since the occurrence of any one of $\Cc_{A^s_i}(\Sc_n^+)$, $i = 0, \ldots ,k,$ implies the occurrence of $\Cc_{A}(\Sc_n^+)$, we have the corollary.
\end{proof}

\begin{proof}[Proof of Proposition \ref{p:qi}]
In what follows we prove \eqref{eq:cross prob asymp ind}; the proof of \eqref{eq:cross prob asymp ind2} is essentially identical. Fix $c > 1$ and $\varepsilon > 0$ and take $C$ and $n$ sufficiently large that the conclusion of Proposition~\ref{p:meas} holds, and
\begin{equation}
\label{eq:cross prob asymp ind3}
  \sup_{ \substack{ x, y \in X , \\ d(x, y) >  C s } } \left(d(x,y)s_{n}^{-1}\right)^{18 + \eta}    | \kappa_n(x, y) | <  \varepsilon^{10 + \eta/3} ;
  \end{equation}
this latter is possible by \eqref{eq:rapid decay corr}.

Now let $s > C s_n$, subsets $X_1, X_2 \subset X $ such that $d(X_1, X_2) > s$, and polygons $P_1 \in \rm{Poly}_{X_1;c}(s)$ and $P_2 \in \rm{Poly}_{X_2;c}(s)$ be given. By Proposition~\ref{p:meas}, there exists a number $c_1 > 0$, independent of $\varepsilon, s, P_1$ and $P_2$, such that the events $\mathcal{C}_{P_1}(\mathcal{S}^+_n)$ and $\mathcal{C}_{P_2}(\mathcal{S}^-_n)$ are determined, outside an event of probability less than $\varepsilon$, by the signs of sets $\mathcal{P}_1 \subset X_1$ and $\mathcal{P}_2 \subset X_2$ respectively, each of cardinality at most
\begin{equation*}
|\mathcal{P}_1|,\, |\mathcal{P}_2| < c_1 \varepsilon^{-2} (s/s_n)^{6} .
\end{equation*}
Applying the first statement of Lemma \ref{l:comp} to compare between the joint law on one hand and the product laws on the other hand for the field restricted on $\mathcal{P}_1 \cup \mathcal{P}_2$, we have, for some constant $c_2 > 0$ independent of $\varepsilon, s, P_1$ and $P_2$,
\begin{equation*}
\label{eq:diff prob prod small}
\begin{split}
&   \left| \mathbb{P} \left(\mathcal{C}_{P_1}(\mathcal{S}^+_n) \cap \mathcal{C}_{P_2}(\mathcal{S}^-_n) \right) - \mathbb{P}(\mathcal{C}_{P_1}(\mathcal{S}^+_n))\cdot \mathbb{P}( \mathcal{C}_2(\mathcal{S}^-_n) )  \right|  \\
&  \qquad  < \varepsilon +  c_2  \bigg( \varepsilon^{-6 - \eta/3}  (s/s_n)^{18 + \eta} \sup_{ \substack{ x, y \in X , \\ d(x, y) >  s } } | \kappa_n(x, y) | \bigg)^{1/4}  \\
&   \qquad < \varepsilon +  c_2  \bigg( \varepsilon^{-6 - \eta/3}   \sup_{ \substack{ x, y \in X , \\ d(x, y) >  C s_n } } (d(x, y) s_n^{-1})^{18 + \eta} | \kappa_n(x, y) | \bigg)^{1/4}  \\
& \qquad < \varepsilon + c_2 \varepsilon .
\end{split}
\end{equation*}
where in the last line we used \eqref{eq:cross prob asymp ind3}. Since $\varepsilon > 0$ was arbitrary, we conclude the proof.
\end{proof}

\subsubsection{Perturbation on macroscopic scales}
\label{s:macro}

The second consequence of the perturbation analysis is controlling the perturbations on macroscopic scales, the key step in handling asymptotically negligible negative correlations.

\begin{proposition}
\label{p:neg}
Let $\eta > 0$ and fix a sequence $p_n > 0$ of positive numbers satisfying
\begin{equation}
\label{eq:cn decay}
\lim_{n \to \infty}  p_n s_n^{-12 - \eta}    = 0.
\end{equation}
Define the sequence of centred Gaussian random fields $(\tilde{f}_n)_{n \in \mathbb{N}}$ on $\mathbb{X}$ with respective covariance kernels
\[ \tilde{\kappa}_n = \kappa_n + p_n ;\]
this is a valid covariance kernel since the constant function is positive-definite. Let $\tilde{\mathcal{S}}^+_n$ denote the positive excursion set of $\tilde{f}_n$. Then for every $c > 0$,
\[   \lim_{n \to \infty} \,  \sup_{s >0 } \, \sup_{P \in \rm{Poly}_{\mathbb{X};c}(s)}  | \mathbb{P}(\mathcal{C}_P(\mathcal{S}^+_n)) - \mathbb{P}(\mathcal{C}_P(\tilde{\mathcal{S}}^+_n)) |    = 0 . \]
\end{proposition}

\begin{proof}
Fix $c > 0$ and $\varepsilon > 0$, and take $n$ sufficiently large that the conclusion of Proposition~\ref{p:meas} holds, and
\begin{equation}
\label{eq:cn decay2}
 s_n^{-12 - \eta} p_n < \varepsilon^{8 + \eta/3},
 \end{equation}
possible by \eqref{eq:cn decay}.
Now let $s > 0$ and $P \in \rm{Poly}_{\mathbb{X}; c}(s)$ be given. Observe that the sequence of covariance kernels $\tilde{\kappa}_n$ also converge locally uniformly on $B(\delta_0)$, along with their first four derivatives, to the same limit $K_\infty$. By Proposition \ref{p:meas}, there exists a number $c_1 > 0$, independent of $\varepsilon, s$ and $P$, such that the events $\mathcal{C}_P(\mathcal{S}^+_n)$ and $\mathcal{C}_P(\tilde{\mathcal{S}}^+_n)$ are determined, outside an event of probability less than $\varepsilon$, by the signs of a set $\mathcal{P} \subseteq P$ of cardinality at most
\[ c_1 \varepsilon^{-2} s_n^{-6} ; \]
for this recall that $\mathcal{P}$ can be chosen to be the same set for all $\kappa_n$ that converge locally uniformly on $B(\delta_0)$ to the same limit $K_\infty$ (see the comments after the statement of Proposition~\ref{p:meas}). Applying the second statement of Lemma \ref{l:comp} to the law on $\mathcal{P}$ of the fields $f_n$ and $\tilde{f}_n$ respectively, for some constant $c_2 > 0$ independent of $\varepsilon, s$ and $P$
\begin{align*}
 \left| \mathbb{P} \left(\mathcal{C}_P(\mathcal{S}^+_n) \right) - \mathbb{P}(\mathcal{C}_P(\tilde{\mathcal{S}}^+_n))   \right|   < \varepsilon +  c_2  \bigg( \varepsilon^{-4 - \eta/3} s_n^{-12 - \eta} p_n  \bigg)^{1/4}  < \varepsilon +  c_2 \varepsilon ,
\end{align*}
where to obtain the last inequality we used \eqref{eq:cn decay2}. Since $\varepsilon > 0$ was arbitrary, we conclude the proof.
\end{proof}

\subsection{Concluding the proof of Theorem \ref{t:main}}
\label{sec:main thm concl}

We are now almost ready to conclude the proof of Theorem \ref{t:main}. Before we begin, we state some simple geometric lemmas and show how to verify the `microscopic' conditions (1) and (5) of Theorem \ref{t:rsw} in the setting of Theorem~\ref{t:main}.

We work in the same set-up as for the perturbation analysis given at the beginning of section \ref{s:proof}. Recall that $\mathcal{S}_n^+$ and $\mathcal{S}_n^-$ denote, respectively, the positive and negative excursion sets of $f_n$; we denote by $\mathcal{N}_n$ the nodal set of $f_n$.

\subsubsection{Geometric lemmas}

In the proof of Theorem \ref{t:main} we shall need the following. Recall the definition of polygons in Definition \ref{d:poly}.

\begin{lemma}
\label{l:geom}
Fix $X \subseteq \mathbb{X}$ and $c > 0$. Then there exists a number $c_1 > 0$ such that for each $s > 0$ and quad $Q\in \rm{Unif}_{X;c}(s)$ the following hold:
\begin{enumerate}
 \item There exists a polygon $P \in \rm{Poly}_{X;c_1}(s) \cap \rm{Unif}_{X;c}(s)$ such that the event $\Cc_P(\Sc^+_n)$ implies the event $\Cc_Q(\Sc^+_n)$.
 \item There exist disjoint domains $X_1, X_2 \subset X$ satisfying $d(X_1, X_2) > s/c_1$ and polygons $P_1 \in \rm{Poly}_{X_1;c_1}(s/c_1) \cap \rm{Unif}_{X;c_1}(s/c_1)$ and $P_2 \in \rm{Poly}_{X_{2};c_1}(s/c_1) \cap \rm{Unif}_{X;c_1}(s/c_1)$ such that if the events $\Cc_{P_1}(\Sc^+_n)$ and $\Cc_{P_2}(\Sc^-_n)$ both hold, then so does $\Cc_{Q}(\mathcal{N}_n)$.
 \end{enumerate}
\end{lemma}

\begin{figure}
\centering
\includegraphics[width=0.48\textwidth]{box_chain.pdf}
\includegraphics[width=0.48\textwidth]{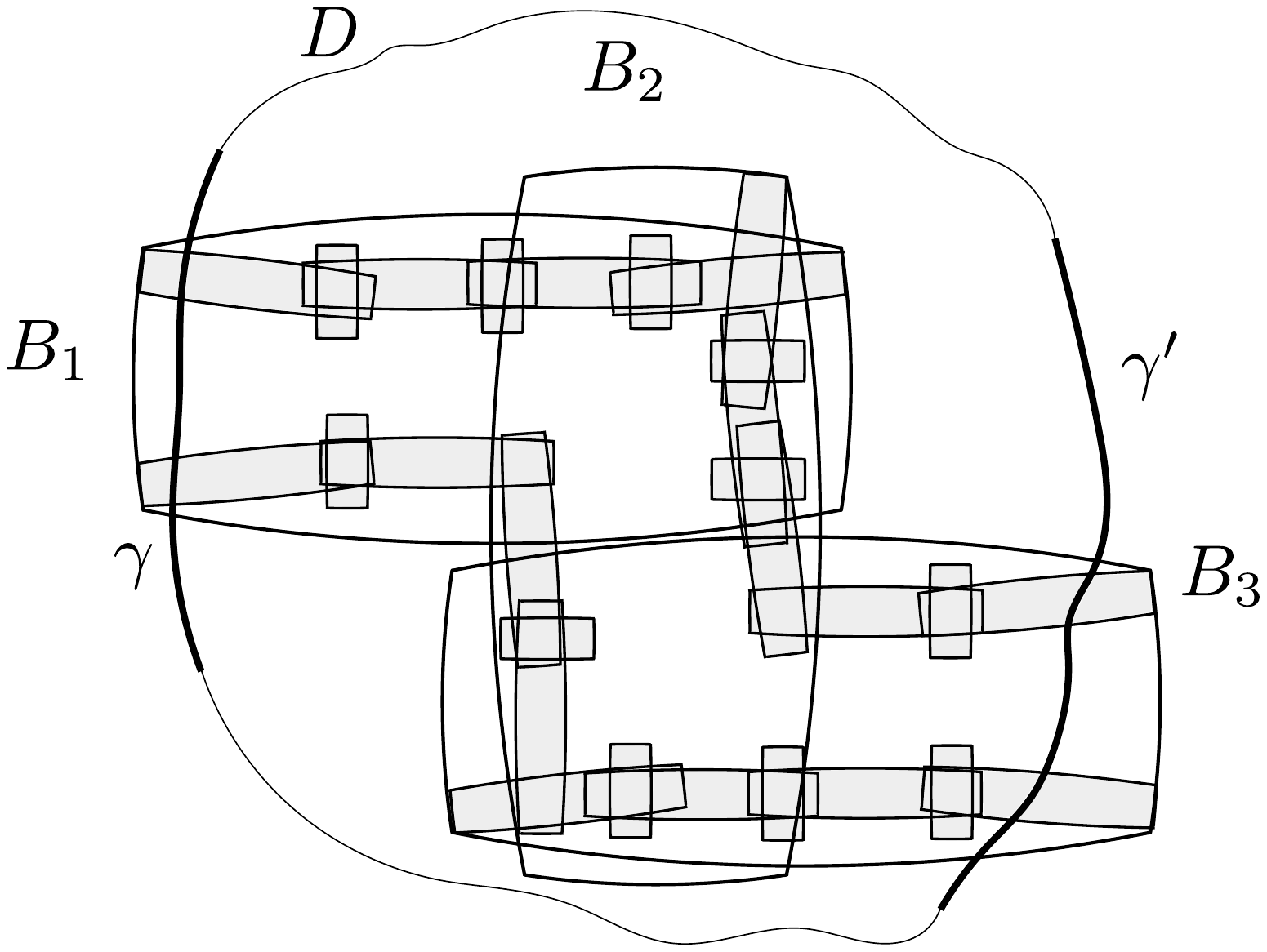}
\caption{Left: The union of a box-chain that crosses a quad $Q = (D,\gamma,\gamma')$ forms a polygon such that a crossing of the polygon implies a crossing of the quad. Right: Within this polygon one can choose two well separated box chains of smaller rectangles such that a crossing of these by $\mathcal{S}_n^+$ and $\mathcal{S}_n^-$ respectively implies a crossing by $\mathcal{N}_n$ of the polygon.}
\label{fig:box chain nodal}
\end{figure}

\begin{proof}
For the first statement of Lemma \ref{l:geom}, one can simply take the polygon that is the union of the boxes comprising one of the box-chains that cross $Q$ guaranteed by the Definition~\ref{d:unifcross} (see Figure \ref{fig:box chain nodal}, left). For the second statement of Lemma \ref{l:geom}, we observe that the statement is true for any box $B  \in \rm{Unif}_{X;c}(s)$, since $B$ can be `divided' along two well-spaced geodesics into three parts, and the top and bottom parts can be crossed by box-chains using smaller boxes. Then for any quad $Q  \in \rm{Unif}_{X;c}(s)$ we can take the box-chain that crosses~$Q$ and decompose each constituent box using these geodesics (see Figure \ref{fig:box chain nodal}, right).
\end{proof}

\subsubsection{Verifying the microscopic conditions}

Here we argue that the two `microscopic' conditions (1) and (4) of Theorem
\ref{t:rsw} are satisfied; we begin by verifying the non-degeneracy condition (1). For use in this subsection we introduce $F_n$ (resp.\ $F_\infty$) as the centred, unit variance Gaussian random field on $\mathbb{R}^2$ with
the rescaled covariance kernel~$K_n$ (resp.~$K_\infty$), as in Definition \ref{def:loc unif scal}.

\begin{lemma}
\label{l:cond1}

~\begin{enumerate}
\item

For every $n \in \mathbb{N}$, $\mathbb{P}( f_n(0) = 0 ) = 0$.

\item Recall that for $v \in \mathbb{S}^1$ and $r>0$ the $\mathcal{L}_v(r)$ is the length-$r$ geodesic segment based at
the origin in direction $v$. For every $v \in \mathbb{S}^1$ we have
\begin{equation*}
\lim_{r \to 0} \, \limsup_{n \to \infty} \, \mathbb{P}(\exists x \in \mathcal{L}_v(r s_n).\,   f_n(x) = 0  ) = 0 .
 \end{equation*}

\end{enumerate}

\end{lemma}
\begin{proof}
The first statement is clear upon recalling that $f_n$ is symmetric and non-degenerate. For the second statement, observe that by the Kac-Rice formula \cite[Theorem 6.3]{AW}, the symmetries of $f_n$, and since Assumption \ref{a:nondegen} guarantees that $\nabla f_n(0)$ is independent of $f_n(0)$, for each $r > 0$,
\[ \mathbb{E} [  | \{ x \in \mathcal{L}_v(r s_n) :  f_n(x) = 0 \} | ] =     \frac{r}{\sqrt{2 \pi}}  \,   \, s_n \,  \mathbb{E}\left[  \left| \frac{\partial f_n(0)}{ \partial v} \right|  \right]  . \]
Since $\kappa_n$ is $C^2$, it holds that
\[ \mathbb{E}\left[   \left| \frac{\partial f_n(0)}{ \partial v} \right|  \right]   =     \sqrt{ \frac{2}{\pi}  \frac{    - \partial^2 \kappa_n(0)}{ \partial^2 v} } . \]
Hence by the local uniform convergence of the second derivatives of $\kappa_n$, and since $F_\infty$ satisfies Assumption~\ref{a:nondegen},
\[    \lim_{n \to \infty}  \, s_n \, \mathbb{E}\left[   \left| \frac{\partial f_n(0)}{ \partial v} \right|  \right]   =   \lim_{n \to \infty}  \,  \sqrt{ \frac{2}{\pi}  \frac{    - s_n^2 \partial^2 \kappa_n(0)}{ \partial^2 v} }    = \sqrt{ \frac{2}{\pi}  \frac{   -  \partial^2 K_\infty(0)}{ \partial^2 v} }  < \infty .\]
Taking $r \to 0$ yields the result.
\end{proof}

Next we verify condition (4) of Theorem \ref{t:rsw} guaranteeing arbitrary crossings on microscopic scales; to this end we formulate the following lemma.

\begin{lemma}
\label{l:exp}
There exists a number $\delta > 0$ such that
\begin{equation*}
\liminf_{n \to \infty}  \mathbb{P} \left(F_n(x) > 0 \text{ for all }  x \in B(\delta) \right) > 0 .
\end{equation*}
\end{lemma}

Before we state the proof of Lemma \ref{l:exp}, we show that it implies condition (4) of Theorem~\ref{t:rsw}.

\begin{corollary}
\label{c:cont}
There exists a number $\delta > 0$ such that
\[\liminf\limits_{n\rightarrow\infty} \inf\limits_{Q \in\rm{Quad}_{B(\delta s_n)}} \, \mathbb{P} ( \mathcal{C}_Q(\mathcal{S}^+_n) ) > 0 ,\]
where $\Sc^{+}_{n}$ are the positive excursion sets \eqref{eq:Sn excursion def} of $f_{n}$.
\end{corollary}
\begin{proof}
By Lemma \ref{l:exp}, there is a number $\delta > 0$ such that
 \[ \liminf_{n \to \infty}  \mathbb{P} \left(  \Phi( B(\delta s_n ) )  \subseteq \mathcal{S}^+_n\   \right) > 0 .  \]
Since $\Phi$ is locally an isometry, for any $\delta_1 < \delta$, $ \Phi( B(\delta s_n ) )$ eventually contains the disk $B(\delta_1 s_n)$. Finally, since the occurrence of the event $\{B(\delta_1 s_n)  \subseteq \mathcal{S}^+_n\}$ implies the crossing event $\Cc_Q(\Sc_n)$ for any $Q \subset B(\delta_1 s_n)$, we have the result.
\end{proof}

\begin{proof}[Proof of Lemma \ref{l:exp}]
Recall that $K_\infty$ denotes the limit of $K_n$, well-defined as a stationary $C^2$ covariance kernel on $B(\delta_0)$. Define a stationary covariance kernel $\tilde{K}_\infty$ on $B(\delta_0/2)$ by
\[  \tilde{K}_\infty(x,y) = K_\infty(x/2, y/2) ,\]
and let $\tilde{F}_\infty$ denote the centred, unit variance Gaussian random field on $B(0, \delta_0/2)$ with covariance kernel~$\tilde{K}_\infty$. By the local uniform convergence of $K_n$ and its first three derivatives to~$K_\infty$ and its respective derivatives,
and the strictly negative second derivatives of~$K_\infty$ (since~$K_\infty$ satisfies Assumption \ref{a:nondegen}), there exists a $\delta_1 \in (0, \delta_0 / 2)$ such that, for sufficiently large $n$,
\[ K_n(x, y) > \tilde{K}_\infty(x, y)  \qquad \text{for all}  \quad x,y \in B(\delta_1). \]
Hence by Slepian's lemma \cite[Theorem 2.2.1]{Adler-Taylor}, for sufficiently large $n$, every $\delta \in (0, \delta_1)$ satisfies
\begin{equation}
\label{eq:prob pers comparison}
\mathbb{P} \left( F_n(x) > 0 \text{ for all }  x \in B(\delta)\right)
\ge
\mathbb{P} \left(\tilde{F}_\infty(x)  \text{ for all }  x \in B(\delta) \right).
\end{equation}
It then remains to prove the existence of a $\delta \in (0, \delta_1)$ such that the
latter probability is positive.

By the Borel-TIS Theorem \cite[Theorem 2.1.1]{Adler-Taylor} and Markov's inequality, there exists a number $c_1 > 0$ such that for every $\lambda > 0$,
\[   \mathbb{P} \left(    \sup_{v \in \mathbb{S}^1} \max_{ x \in B(\delta_0/2)  } \left|\frac{\partial \tilde{F}_{\infty}}{\partial v}(x)\right|  > \lambda \right) < c_1/ \lambda .  \]
Hence, by taking $\lambda_1, \lambda_2 > 0$ sufficiently small, the event
\[ E=  \left\{   \tilde{F}_\infty(0) > \lambda_1\right\} \cap\left\{\sup_{v \in \mathbb{S}^1} \max_{x \in B(\delta_0/2)} \left|\frac{\partial \tilde{F}_{\infty}}{\partial v}(x)\right|  < \lambda_2  \right\}  \]
has positive probability. By Taylor's theorem we can choose $\delta > 0$ sufficiently small that
\[
E\subseteq  \{ \tilde{F}_\infty(x)  \text{ for all }  x \in B(\delta) \}  ;
 \]
which, since $\mathbb{P}(E) > 0$ and in light of \eqref{eq:prob pers comparison}, yields Lemma \ref{l:exp}.
\end{proof}

\subsubsection{Proof of Theorem \ref{t:main} assuming Theorem \ref{t:rsw}, Proposition \ref{p:meas} and Lemma \ref{l:comp}}

Let $(f_n)_{n \in \mathbb{N}}$ be given as in Theorem \ref{t:main}; with no loss of generality we assume that $f_n$ are unit variance. Let $\eta > 0$, $X \subset \mathbb{X}$ and $s_n$ satisfy the conditions of Theorem \ref{t:main}.

We begin by slightly perturbing the covariance kernels $\kappa_n$ of $f_n$ to eliminate possible negative correlations. Define a collection of centred Gaussian random fields $(\tilde{f}_n)_{n \in \mathbb{N}}$ on $\mathbb{X}$ with respective covariance kernels
\begin{equation}
\label{eq:kappatild pert kappa}
\tilde{\kappa}_n (x, y) =   \kappa_n(x, y) + s_n^{12 + \eta/2}.
\end{equation}
Observe that, by condition \eqref{eq:asymp pos}, $\tilde{\kappa}_n$ is everywhere positive on $X$ for $n$ sufficiently large. Moreover, the choice
\eqref{eq:kappatild pert kappa} of perturbation means that the conclusion of Proposition \ref{p:neg} is valid.

We now argue that the positive excursion sets $\tilde{\mathcal{S}}^+_n$ of $\tilde{f}_n$ satisfy all the conditions of Theorem~\ref{t:rsw} for the set $X$ and sequence $s_n$; by symmetry, the same conclusion holds also for the negative excursion sets.
The justification for the validity of conditions (2), (3) and (5) of Theorem \ref{t:rsw} is via standard arguments: the symmetry of the excursion sets follows from the symmetry of the kernel, positive associations on $X$ follow from the positivity of the covariance kernels on~$X$ by the well-known result of Pitt \cite{Pitt}, and the probability of crossing square-boxes is exactly $1/2$ by the symmetry of the kernel and the symmetry of a Gaussian random field w.r.t.\ sign changes.
Moreover, conditions (1), (4) and (6) in Theorem \ref{t:rsw} follow from the analysis we developed above, namely Lemma \ref{l:cond1}, and corollaries \ref{c:cont} and~\ref{c:qi} respectively. Hence all the conditions of Theorem~\ref{t:rsw} are satisfied, and an application of Theorem \ref{t:rsw} yields the desired conclusions for $\tilde{f}_n$, i.e.\ that the positive (resp.\ negative) excursion sets of~$\tilde{f}_n$ satisfy the RSW estimates on $X$ on all scales. In particular, for all $c > 0$,
\begin{equation}
\label{e:rswtilde}
   \liminf_{n \to \infty} \,  \inf_{s > 0} \, \inf_{Q \in \rm{Unif}_{X;c}(s) }  \mathbb{P}( \Cc_Q(\tilde{ \mathcal{S}}_n^+) ) > 0 .
\end{equation}

Next we use Proposition \ref{p:neg} to infer that the positive excursion sets of $f_n$ also satisfy the RSW estimates on $X$ on all scales (the same statement for the negative excursion sets $\mathcal{S}^-_n$ then holds by an identical argument). Fix $c > 0$, and let $c_1 > 0$ be the constant prescribed by Lemma \ref{l:geom}. Also let $\varepsilon > 0$ be such that, for all sufficiently large $n \in \mathbb{N}$, both
\begin{equation}
\label{e:rswtilde2}
   \inf_{s > 0} \, \inf_{Q \in \rm{Unif}_{X;c}(s) }  \mathbb{P}( \Cc_Q(\tilde{\mathcal{S}}_n^+) )  > 2\varepsilon  ,
   \end{equation}
and
\begin{equation}
\label{e:rswtilde3}
   \sup_{s >0 } \, \sup_{P \in \rm{Poly}_{X;c_1}(s)}  | \mathbb{P}(\mathcal{C}_P(\mathcal{S}^+_n)) - \mathbb{P}(\mathcal{C}_P(\tilde{\mathcal{S}}^+_n)) |    <  \varepsilon ,
   \end{equation}
hold; possible by \eqref{e:rswtilde} and Proposition \ref{p:neg} respectively. Now let $ s > 0$ and $Q \in \rm{Unif}_{X; c}(s)$ be given. By Lemma \ref{l:geom}, there exists a polygon $P \in \rm{Poly}_{X;c_1}(s) \cap \rm{Unif}_{X;c}(s)$ such that the event $\Cc_P(\Sc_{n}^{+})$ is contained in the event $\Cc_Q(\Sc_{n}^{+})$. In particular, since $P \in  \rm{Unif}_{X;c}(s)$, by \eqref{e:rswtilde2}
\[     \mathbb{P}(\Cc_P(\tilde{\Sc}_{n}^{+})) > 2 \varepsilon   ,\] and in light of
\eqref{e:rswtilde3}, applicable since $P \in \rm{Poly}_{X;c_1}(s)$, we obtain
\[   \mathbb{P}(\mathcal{C}_P(\mathcal{S}^+_n))  > \varepsilon  .  \]
Finally, since $\Cc_P(\Sc_{n}^{+})\subseteq \Cc_Q(\Sc_{n}^{+})$, we conclude that
\[   \mathbb{P}(\mathcal{C}_Q(\mathcal{S}^+_n)) \ge \mathbb{P}(\mathcal{C}_P(\mathcal{S}^+_n))  > \varepsilon  ,  \]
the RSW estimates for $\Sc_n^+$ on all scales.

The final step of the proof of Theorem \ref{t:main} is using the first statement \eqref{eq:cross prob asymp ind} of Proposition~\ref{p:qi} to infer the RSW estimates for the {\em nodal sets} $\mathcal{N}_n$ of $f_n$ from the already established RSW estimates for the {\em excursion sets} of $f_{n}$. Again fix $c > 0$, and let $c_1 > 0$ be the corresponding constant appearing from Lemma \ref{l:geom}. Let $\varepsilon > 0$ and $C > 0$ be such that, for all sufficiently large $n \in \mathbb{N}$,
\begin{equation}
\label{e:rsw2}
   \inf_{s > 0} \, \inf_{ \substack{ Q_1 \in \rm{Unif}_{X;c_1}(s/c_1) , \\ Q_2 \in \rm{Unif}_{X;c_1}(s/c_1) } } \, \mathbb{P}( \Cc_{Q_1}(\Sc_n^+) ) \cdot \mathbb{P}( \Cc_{Q_2}(\Sc_n^-) )   > 2\varepsilon ,
   \end{equation}
and
\begin{equation}
\label{e:rsw3}
\sup_{ s > C s_n }  \, \sup_{ \substack{ X_1, X_2 \subset X , \\ d(X_1, X_2) > s/c_1} } \, \sup_{ \substack{P_1 \in \rm{Poly}_{X_1;c_1}(s/c_1) , \\ P_2 \in \rm{Poly}_{X_2;c_1}(s/c_1) } } \left| \mathbb{P} \left(\mathcal{C}_{P_1}(\mathcal{S}^+_n) \cap \mathcal{C}_{P_2}(\mathcal{S}^-_n) \right) - \mathbb{P}(\mathcal{C}_{P_1}(\mathcal{S}^+_n))\cdot  \mathbb{P}( \mathcal{C}_{P_2}(\mathcal{S}^-_n) )  \right| < \varepsilon.
\end{equation}
both hold; possible since the RSW estimates hold for the excursion sets of $f_n$ on all scales and by  \eqref{eq:cross prob asymp ind} respectively.

Now let $s > Cs_n$ and $Q \in \rm{Unif}_{X; c}(s)$ be given. By Lemma \ref{l:geom} there exist disjoint domains $X_1, X_2 \subset X$ satisfying $d(X_1, X_2) > s/c_1$ and polygons $P_1 \in \rm{Poly}_{X_1;c_1}(s/c_1) \cap \rm{Unif}_{X;c_1}(s/c_1)$ and $P_2 \in \rm{Poly}_{X_2;c_1}(s/c_1) \cap \rm{Unif}_{X;c_1}(s/c_1)$  such that if the events $\Cc_{P_1}(\Sc^+_n)$ and $\Cc_{P_2}(\Sc^-_n)$ both occur, then so does $\Cc_{Q}(\mathcal{N}_n)$,
i.e.,
\[  \Cc_{P_1}(\Sc^+_n)\cap \Cc_{P_2}(\Sc^-_n) \subseteq  \Cc_{Q}(\mathcal{N}_n) . \]
In particular, since $P_1, P_2 \in  \rm{Unif}_{X;c_1}(s/c_1)$, by \eqref{e:rsw2}
\[     \mathbb{P}(\Cc_{P_1}(\Sc^+_n))  \cdot  \mathbb{P}(\Cc_{P_2}(\Sc^-_n))  > 2 \varepsilon  . \]
Since also $P_1 \in \rm{Poly}_{X_1;c_1}(s/c_1)$, $P_2 \in \rm{Poly}_{X_2;c_1}(s/c_1)$ and $d(X_1, X_2) > s/c_1$, in light of \eqref{e:rsw3} we deduce that
\[   \mathbb{P}(\mathcal{C}_{P_1}(\mathcal{S}^+_n) \cap \mathcal{C}_{P_2}(\mathcal{S}^-_n))  > \varepsilon  .  \]
Finally, since $\Cc_{P_1}(\Sc^+_n) \cap \Cc_{P_2}(\Sc^-_n)\subseteq\Cc_{Q}(\mathcal{N}_n)$, we conclude that
\[   \mathbb{P}(\mathcal{C}_Q(\mathcal{N}_n)) \ge   \mathbb{P}(\mathcal{C}_{P_1}(\mathcal{S}^+_n) \cap \mathcal{C}_{P_2}(\mathcal{S}^-_n))   > \varepsilon  ,  \]
which validates the RSW estimates for $\mathcal{N}_n$ down to the scale $s_n$.

\subsection{Proof of Theorem \ref{t:kos} and the validity of Example \ref{ex:torus}}
\label{sec:prf thm Kostlan}

In this section we show that Theorem \ref{t:kos} and Example \ref{ex:torus} are within the scope of the more general Theorem \ref{t:main}.

\begin{proof}[Proof of Theorem \ref{t:kos}]
Observe that the covariance kernels $\kappa_n$ are symmetric in sense of Definition \ref{a:symmetry} and satisfy Assumption \ref{a:nondegen}.  Next we check the local uniform convergence of $\kappa_n$ together with all its derivatives on the scale~$s_n$ (previously stated at \eqref{eq:Kx(u,v)->Kinfty(u-v)}). For this, define the smooth functions $G_n: \mathbb{R}^2 \times \mathbb{R}^2 \to \mathbb{R}$ and $F_n: \mathbb{R} \to \mathbb{R}$ by
\[ G_n(x, y) =   \sqrt{n}   \cdot \| \Phi( x/\sqrt{n}) - \Phi(y/\sqrt{n} )  \| \quad \text{and} \quad  F_n(t) =  (\cos(t/n))^n .\]
An explicit computation shows that $G_n$ and $F_n$ converge locally uniformly together with all of their derivatives to the respective limits
\[  G_\infty(x, y) = \|x-y\|  \quad \text{and} \quad F_\infty(t) = e^{-t^2/2} ,\]
and hence so does their composition $F_n \circ G_n$. Since
\[  K_n(x,y) = F_n \circ G_n( x, y)  = \kappa_n( \Phi(s_n x), \Phi(s_n, y)  )  \quad \text{and} \quad K_\infty(x, y) = F_\infty \circ G_\infty(x, y)  , \]
we have the stated convergence.

It remains to show that conditions (4) and (5) of Theorem \ref{t:main} hold for any constant $\eta > 0$, scale $s_n = n^{-1/2}$ and $X \subset \mathbb{S}^2$ whose closure does not contain antipodal points.
For condition~(4), we observe that since the closure of $X$ does not contain antipodal points, there exists a number $c_1 < \pi$ such that $\theta(x, y) < c_1 $ for each $x,y \in X$. Therefore there exists a $c_2 > 0$ such that, for all $n \in \mathbb{N}$ and $x,y \in X$,
\[  \kappa_n(x,y ) = \cos( \theta(x, y))^n >  - e^{-c_2n}  ,  \]
and so, for any $\eta > 0$, as $n \to \infty$,
\[  s_n^{-12 - \eta}   \inf_{x,y \in X} (\kappa_n(x,y) \wedge 0 ) = - n^{6 + \eta/2} e^{-c_2n} \to 0  \,. \]

For condition (5), let $c_1 < \pi$ be as above, and choose a $c_2 \in (0, 1/ c_1^2)$ such that $|\cos(t)| \le 1 - c_2 t^2 $ for each $|t| < c_1$. Together with the inequality $\log(1-x) \le -x$, valid on $x \in (0, 1)$, we have for all $x,y \in X$,
\begin{align*}
  |\kappa_n(x, y) | =   |\cos( \theta(x, y) ) |^n  \le     e^{n \log( 1 - c_2 \theta(x, y)^2  )}  \le     e^{- n c_2 \theta(x, y)^2 }  = e^{-c_2 ( d(x, y) s_n^{-1} )^2 }  .
  \end{align*}
  Hence for any $\eta > 0$ and $C > 1$,
  \begin{align*}
     \limsup_{n \to \infty}  \sup_{\substack{x,y \in X, \\ \theta(x, y) > C s_n }} (\theta(x, y)s_n^{-1})^{18 + \eta} \,  |\kappa_n(x, y) |   \le  \limsup_{n \to \infty}  \,  \sup_{t > C}  \, t^{18 + \eta} \, e^{-c_2 t^2 }  = \sup_{t > C}  \, t^{18 + \eta} \, e^{-c_2 t^2 },
         \end{align*}
        which tends to zero as $C \to \infty$.
 \end{proof}

\begin{proof}[Validity of Example \ref{ex:torus}]
Remark first that $\kappa_n$ is a valid covariance kernel since $\cos^n(x) \cos^n(y)$ can be written as a Fourier series $\sum_{i,j} a_{i, j} \cos(ix) \cos(jy)$ for non-negative coefficients $a_{i, j}$, which implies that $\kappa_n$ is positive-definite.

Similarly to in the proof of Theorem \ref{t:kos} above, it is sufficient that the conditions of Theorem~\ref{t:main} hold for any constant $\eta > 0$, scale $s_n = n^{-1/2}$ and subset $X \subseteq \mathbb{T}^2$ such that the closure of $X$ does not contain distinct $x,y \in X$ having $2(x_1 - y_1)$ and $2(x_2 - y_2)$ as integers. The proof of this is similar to the proof of Theorem \ref{t:kos}, so we omit the details.
\end{proof}


\smallskip
\section{Proof of Theorem \ref{t:rsw}: RSW estimates for sequences of random sets}

\label{s:rsw}

In this section we prove the abstract RSW estimates in Theorem \ref{t:rsw}, following the argument in \cite{Tas16} that established the analogous estimates for planar Voronoi percolation.
For the benefit of a reader familiar with \cite{Tas16}, we explain the four main differences in our setting, and well as briefly describing the necessary modifications to the argument.\begin{enumerate}
\item Recall that Theorem \ref{t:rsw} is stated for either the unit sphere $\mathbb{S}^2$ or the flat torus~$\mathbb{T}^2$. The first difference is due to the non-Euclidean geometry of $\mathbb{S}^2$; indeed, since the interior angles of spherical squares (see Definition \ref{d:box chain}) depend on their scale, many of the simple geometric arguments in \cite{Tas16} fail in the spherical case and need to be derived from scratch or modified significantly. On the other hand, on the flat torus these arguments work as in \cite{Tas16}.

\item Second, we work with a sequence of random sets rather than a single set. Hence we rely on extra `uniform' conditions on the covariance kernels in the statement of Theorem~\ref{t:rsw}, which ensure that all the inputs into the argument are uniformly controlled.

\item Third, the random set considered in \cite{Tas16}, arising from planar Voronoi percolation, is asymptotically independent in a very strong sense: the Voronoi percolation restricted to disjoint domains is independent as long as there are no Voronoi cells intersecting both of them, see the discussion in section~\ref{s:perc}. Since we wish to apply Theorem~\ref{t:rsw} to Gaussian random fields, we do not have this type of strong mixing of the model, and instead we work with a much weaker notion of asymptotic independence (see condition (6) in the statement of Theorem \ref{t:rsw}).

\item Finally, unlike \cite{Tas16}, the property of `positive associations' only applies inside a subset $X \subseteq\mathbb{X}$; this is essential in order to include the Kostlan ensemble \eqref{eq:fn Kostlan def}. As a result, we need to take extra care in the argument to ensure our geometric constructions take place exclusively in this set.
\end{enumerate}

Before embarking on the proof of Theorem \ref{t:rsw}, we first build up a collection of preliminary results that hold for arbitrary $n \in \mathbb{N}$. The first result (Lemma \ref{l:cor13}) can be viewed as a modification of the `standard theory' of RSW: this shows how to transform the bounds on the probability of crossing a small fixed box to infer the bounds on the probability of crossing large domains. The second set of results
(section \ref{sec:Tassion arg}) contains our modification of Tassion's argument in \cite{Tas16}.

Throughout the rest of this section we fix a set $X \subseteq \mathbb{X}$ as in the statement of Theorem~\ref{t:rsw}. Our preliminary results depend only on the conditions of Theorem \ref{t:rsw} that hold for each $n \in \mathbb{N}$, namely the first non-degeneracy statement in condition (1), the symmetry in condition ~(2), and the guarantee of positive associations in $X$ in condition (3). We stress that all the preliminary estimates
that we state give lower bounds on various crossing probabilities depending on $n \in \mathbb{N}$ in terms of a positive power of the quantity
\begin{equation}
\label{e:c0}
 c_0(n) =  \inf_{s >0 } \inf_{ B \in \rm{Box}_{X; 1}(s) } \, \mathbb{P}( \mathcal{C}_{B}(\mathcal{S}_n) );
 \end{equation}
importantly these are monotone increasing in $c_{0}$. By condition (4) of Theorem \ref{t:rsw}, $c_0(n)$ is uniformly bounded away from zero for sufficiently large $n \in \mathbb{N}$, which, in light of the above, yields a uniform control over the crossing probabilities for varying $n$.
For the next two sections we work with arbitrary fixed $n \in \mathbb{N}$, and for notational convenience we drop all dependencies on $n$ and on the random set $\Sc_n$.

\subsection{The `standard theory' of RSW: From a fixed box to larger domains}

One of the most fundamental tools in percolation theory is the FKG inequality, which implies positive associations for the percolation subgraph, and in particular implies that crossing events are positively correlated. In the classical theory (i.e.\ on the plane), the FKG property is used to infer bounds on the probability of crossing larger domains from assumed bounds on the probability of crossing a fixed small box; we call this the `standard theory' of RSW. For instance, in  \cite[Corollary 1.3]{Tas16} `horizontal' crossings of two overlapping rectangles are connected via a `vertical' crossing of a square to deduce a `horizontal' crossing of a longer rectangle.

In our setting the property of positive associations is true in the set $X$ by assumption, and by analogy we shall refer to this fact as the `FKG property'. We next state a version of the `standard theory' of RSW that is valid in the spherical setting. On the sphere, the construction used in \cite[Corollary 1.3]{Tas16} fails, since two spherical rectangles cannot be overlapped in a way that the overlapping region is a square. Instead, we connect `horizontal' crossings using a third `vertical' rectangle.

Let us introduce a fixed box, $\bar B(s)$, which denotes, for each $s > 0$, an $s \times 2s$ box chosen arbitrarily. Recall also that, for each $c, r \ge 1$ and $s > 0$, the collection of boxes and annuli $\rm{Box}_{X; c}(s)$ and $\rm{Ann}_{X; c;r}(s)$ were introduced in definitions~\ref{d:box chain}, \ref{d:box2} and \ref{d:ann}, and note in particular that $\rm{Ann}_{X; 6; 6}(s) $ consists exclusively of $s \times 6s$ annuli.

\begin{lemma}[From a fixed box to arbitrary boxes and annuli; c.f. {\cite[Corollary 1.3]{Tas16}}]
\label{l:cor13}
There exists a sufficiently small $s^\ast > 0$ such that the following holds for every $c > 1$ and $s < s^\ast$:
there exists a monotone increasing function $f_{c}$, depending only on $c$, and an absolute monotone increasing function $g$ such that
\[   \inf_{ B \in \rm{Box}_{X; c}(3s)  } \mathbb{P}(\Cc_B)  > f_{c}   \left( \mathbb{P}(\Cc_{\bar B(s) })  \right)   \quad \text{and} \quad   \inf_{A \in \rm{Ann}_{X; 6; 6}(s)  }  \mathbb{P}( \Cc_A )  > g   \left( \mathbb{P}(\Cc_{\bar B(s)})  \right)   .\]

\end{lemma}

\begin{figure}
\centering
\includegraphics[width=0.48\textwidth]{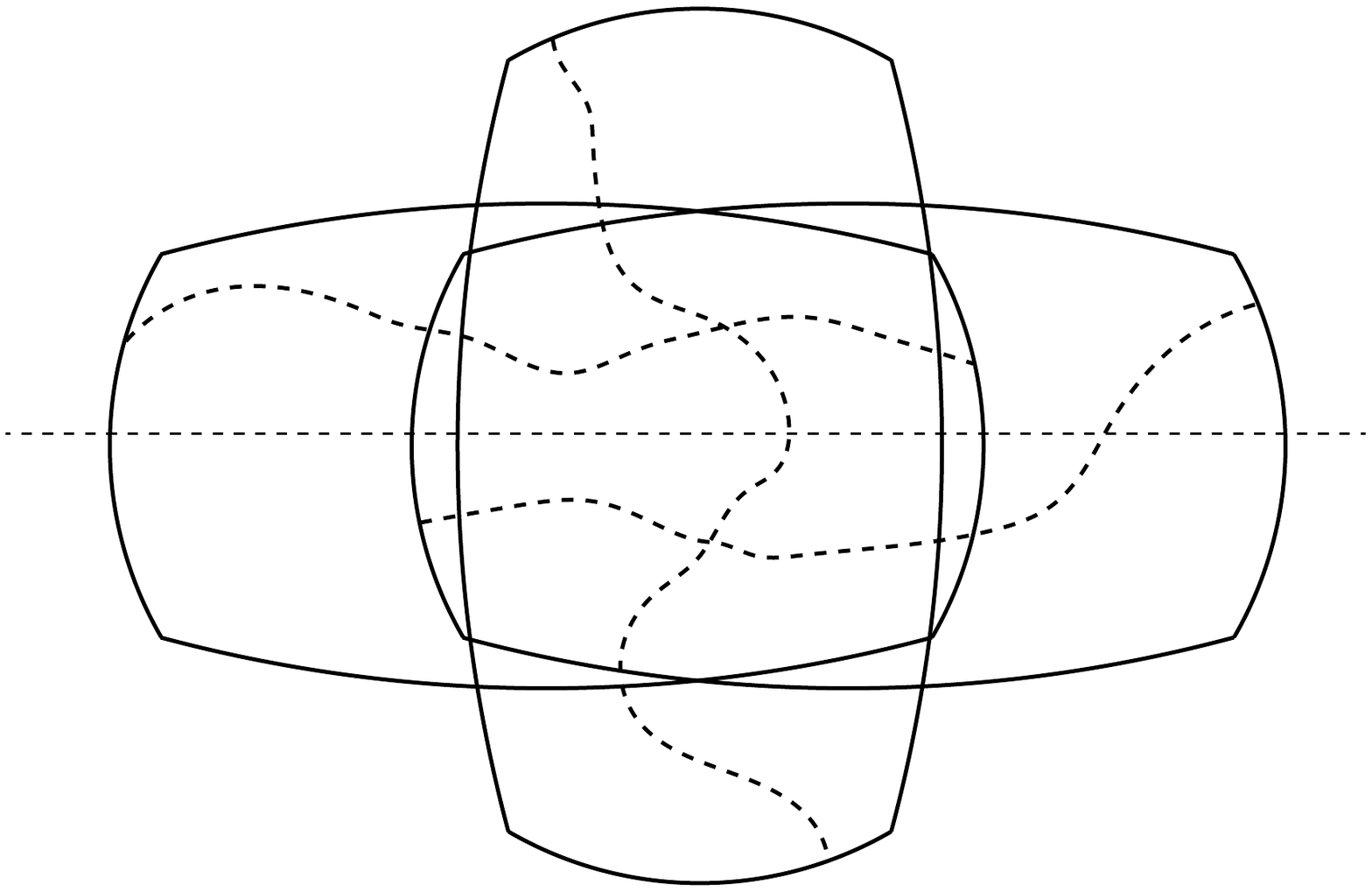}
\includegraphics[width=0.48\textwidth]{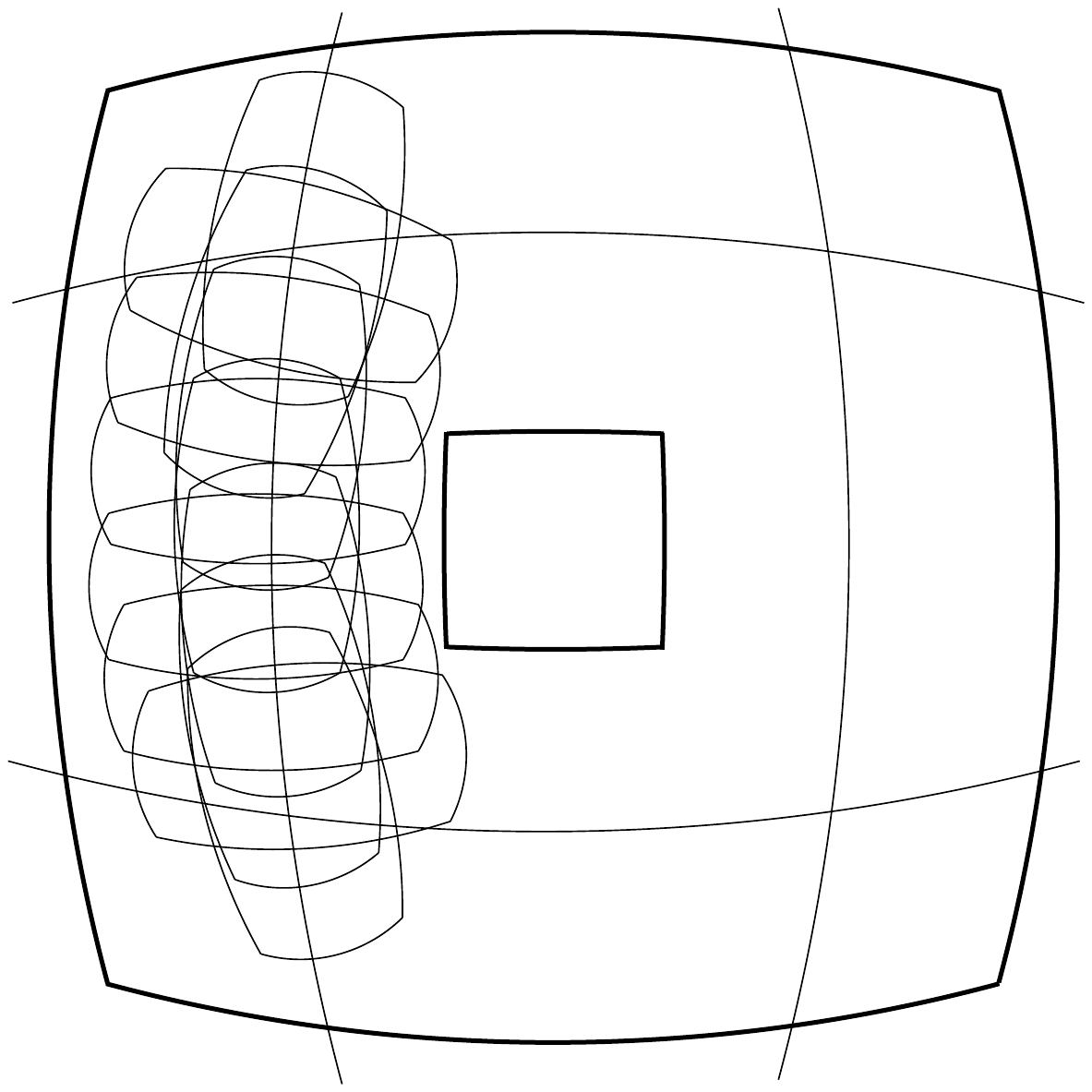}
\caption{Left: A crossing in a `vertical' rectangle connects crossings in two `horizontal' rectangles that are copies of each other shifted along a (dashed) geodesic. Right: By repeating this construction many times we can obtain a long crossing. Combining four of them we obtain a circuit in the annulus. }
\label{fig:gluing two rectangles}
\end{figure}

The value of $s^\ast$ depends only on the geometry of $\mathbb{X}^2$; in the case $\mathbb{X} = \mathbb{T}^2$ it could be arbitrary, whereas in the case $\mathbb{X} = \mathbb{S}^2$ it must be sufficiently small so that the distortions due to the spherical geometry are controlled on a ball $B(s^\ast)$. This constant could be computed explicitly, but its precise value is irrelevant.
In the Euclidean case, the numbers $3$ and $6$ in the statement of Lemma \ref{l:cor13} could be replaced by $2$ and $5$ respectively. Using $3$ and $6$ provides a bit more `space' in the spherical case to account for distortions.

\begin{proof}
The proof is based on the observation that, for sufficiently small $s^\ast > 0$ and $s < s^\ast$, it is possible to form a box-chain out of alternating `horizontal' and `vertical' copies of $\bar B(s)$ that are aligned along a single geodesic (see Figure \ref{fig:gluing two rectangles}, left).

For the first statement, fix $3s \le a, b \le 3cs$ and consider an $a \times b$ box $B \subseteq X$. Let $\{B_i\}$ be a box-chain consisting of horizontal and vertical copies of $\bar B(s)$ aligned along the geodesic joining the mid-points of the opposite sides of $B$. Since the shortest sides of $B$ are longer than the longest sides of $\bar B(s)$, for sufficiently small $s^\ast > 0$ and $s < s^\ast$ we can find such a $\{B_i\}$ that both crosses~$B$ and lies inside $B$ (c.f.\ the Euclidean case, where we could replace the number $3$ with~$2$); moreover, the number of boxes required depends only on $c$. Since the FKG property holds in~$B$, this establishing the bound.

The second statement is proved similarly, working instead with four inter-connecting box-chains aligned along the four `median' geodesics that bisect orthogonally the geodesic line segments joining the mid-points of the boundary squares of any $s \times 6s$ annuli $A$ (see Figure~\ref{fig:gluing two rectangles}, right). Such box-chains can be formed inside $A$ since $\bar B(s)$ fits inside~$A$ when aligned with its shortest sides perpendicular to a geodesic bisecting $A$ (c.f.\ the Euclidean case, where we could replace the number $6$ with~$5$).
\end{proof}

\subsection{Tassion's argument}
\label{sec:Tassion arg}

In this section we develop Tassion's argument from \cite{Tas16}, with suitable modifications to account for the difference in our setting. We begin by introducing, following Tassion, the concept of $H$-crossings and $X$-crossings of square boxes (see Figure \ref{fig:H_s and X_s} for an illustration in the spherical case).

Throughout this section, when the parameter $s^\ast$ in the statement of a lemma may be set sufficiently small, we always implicitly set it so that the conclusion of Lemma \ref{l:cor13}, as well as the conclusion of any proceeding lemmas in this section, is valid.
Since if $X$ has an empty interior Theorem \ref{t:rsw} has no content, we may assume that $X$ has non-empty interior, and, by symmetry, that $X$ contains an open ball $B(\delta_0)$ centred at the origin. Therefore we may assume that $s^\ast$ is sufficiently small so that all of the (finite) collections of domains that we manipulate in the proofs of the following lemmas are contained inside $B(\delta_0)$; we may thereby always assume the FKG property holds.

\begin{definition}

\leavevmode
\begin{enumerate}
\item For each $s > 0$ and $\alpha, \beta \in [0, s/2]$, an $H$-crossing of an $s \times s$ square box $B = (D; \gamma, \gamma')$, denoted by $\Hc_s(\alpha,\beta) = \Hc_{s;B}(\alpha,\beta) $, is the event that a connected component of $\Sc_{n}$, restricted to $B$, intersects both $\gamma$ and the segment of $\gamma'$ of length $\beta-\alpha$ at distance~$\alpha$ from the mid-point of $\gamma'$ (see Figure \ref{fig:H_s and X_s}, left).

\item For every $s > 0$ and $\alpha \in [0, s/2]$, an $X$-crossing of an $s \times s$ square box $B = (D; \gamma, \gamma')$, denoted by $\Xc_s(\alpha) = \Xc_{s;B}(\alpha)$, is the event that a connected component of $\Sc_{n}$, restricted to $B$, intersects the four segments of $\gamma \cup \gamma'$ obtained by removing from each of $\gamma$ and $\gamma'$ the centred intervals of length $2\alpha$ (see Figure \ref{fig:H_s and X_s}, right).
\end{enumerate}
\end{definition}

\begin{figure}[h]
  \centering
  \includegraphics[width=0.48\textwidth]{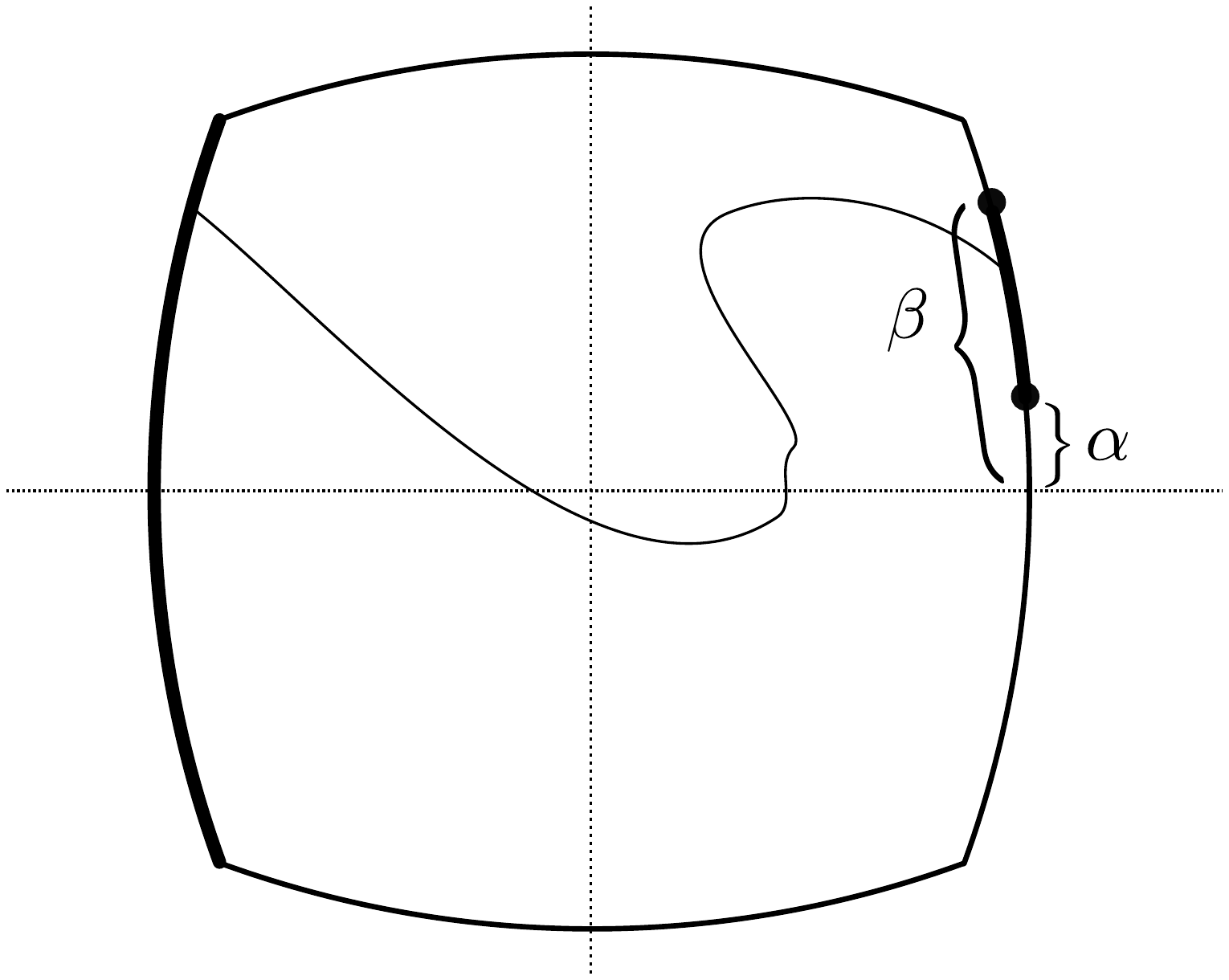}
  \includegraphics[width=0.48\textwidth]{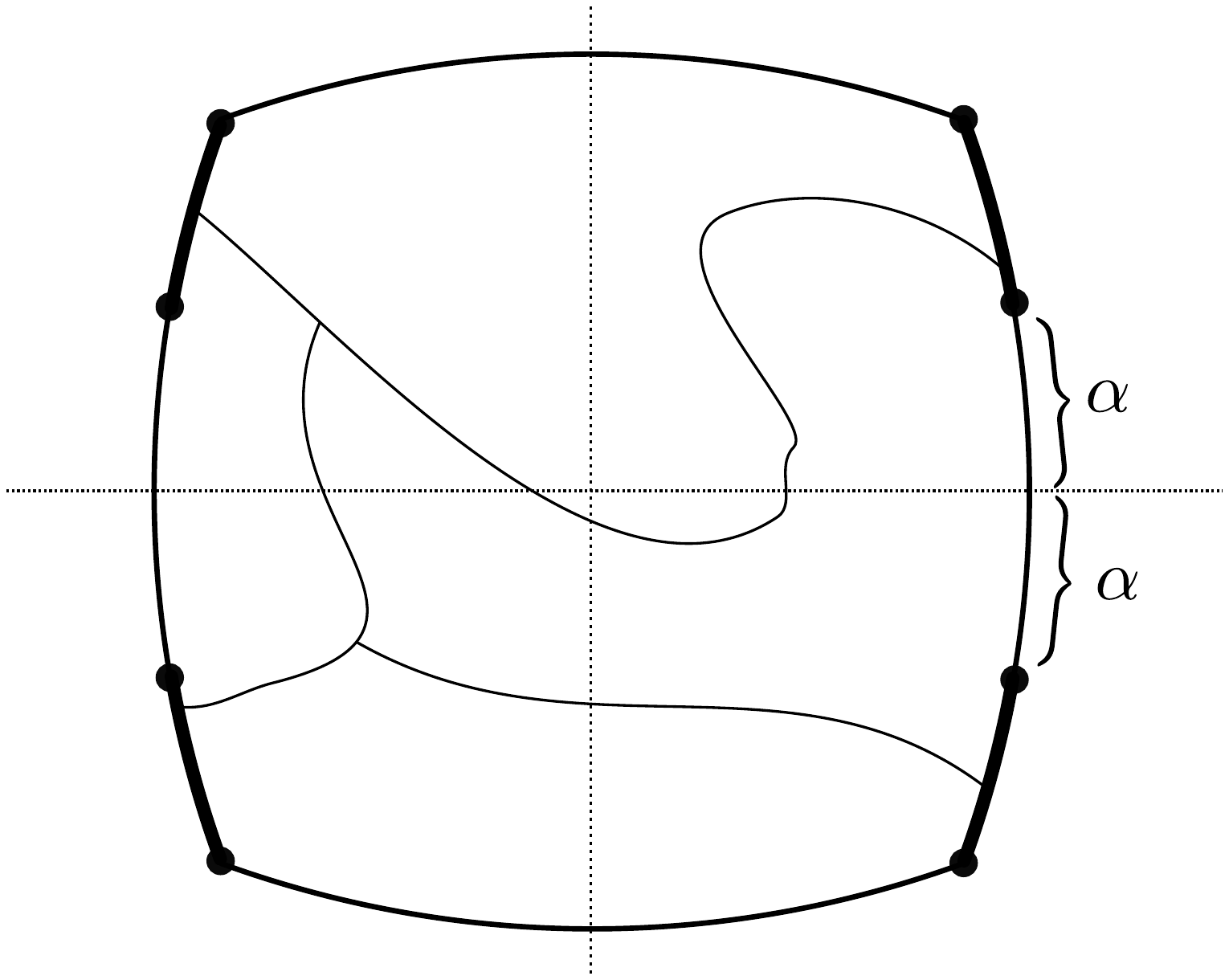}
\caption{An illustration of an $H$-crossing (left) and an $X$-crossing (right) of a square box in the spherical case.}
\label{fig:H_s and X_s}
\end{figure}

Observe that, by the symmetry condition (2) of Theorem \ref{t:rsw} and the definitions of $\Hc_s(\alpha, \beta)$ and $\Xc_s(\alpha) $, both $\P\brb{\Hc_s(\alpha, \beta)}$ and $\P\brb{\Xc_s(\alpha)}$ are independent of the choice of the square box~$B$. Hence the function
\begin{equation}
\label{eq:phisa def}
\phi_s(\alpha)=\P\brb{\Hc_s(0,\alpha)}-\P\brb{\Hc_s(\alpha,s/2)}
\end{equation}
is well-defined, and is a continuous function of~$\alpha$ by the first statement of the non-degeneracy condition (1) of Theorem \ref{t:rsw}.
Recalling the definition \eqref{e:c0} of $c_{0}$, for every scale $s > 0$ we may fix the constant
\[ \alpha_s=\min(\phi^{-1}_s(c_0/4),s/4),  \]
satisfying $\alpha_s \le s/4$.
The following lemma contains the essential consequences of the definition of $\alpha_s$, c.f. \cite[Lemma 2.1]{Tas16}.

\begin{lemma}
\label{lem:Tassion 2.1}
There exists a sufficiently small $s^\ast>0$, and absolute numbers $a_{1}>0$ and $k_{1}\in \mathbb{N}$, such that if $s < s^*$,
the following two properties hold:
\begin{itemize}
\item[{(P1)}] For all $0\le \alpha\le \alpha_s$, $\P\brb{\Xc_s(\alpha)} > a_{1}\cdot c_{0}^{k_{1}}$.
\item[{(P2)}] If $\alpha_s<s/4$, then for all $\alpha_s\le \alpha\le s/2$, $\P\brb{\Hc_s(0,\alpha)} \ge c_0/4+\P\brb{\Hc_s(\alpha,s/2)}$.
\end{itemize}
\end{lemma}

\begin{proof}
The proof of Lemma \ref{lem:Tassion 2.1} is independent of the geometry of the ambient space and the argument from \cite{Tas16} works in our setting unimpaired.
\end{proof}

The next three lemmas are the heart of Tassion's argument. Recall the fixed $s\times 2s$ box $\bar B(s)$ in the statement of Lemma \ref{l:cor13}. We think of $s > 0$ as being a `good' scale if it satisfies $\alpha_s\le 2\alpha_{2s/3}$, and proceed to formulate a few consequences of a good scale. As a corollary, we deduce, for a fixed $n \in \mathbb{N}$, the existence of uniform bounds on crossing probabilities on all large scales, provided that certain inputs into the argument are also controlled.

As in the proof of Corollary \ref{c:qi}, in this section we work with the collection $(A_{a, b})_{a<b}$ of $a \times b$ annuli centred at the origin that are `parallel', i.e.\ such that there is a single geodesic that passes through both mid-points of both pairs of opposite sides. When working with square boxes, we shall sometimes abuse notation by referring to these simply as `squares'.

\begin{lemma}[Good scales imply crossings of the fixed box; c.f. {\cite[ Lemma 2.2]{Tas16}}]
\label{lem:Tassion 2.2}
There exists a sufficiently small $s^* > 0$ and absolute numbers $a_{2}>0$ and $k_{2}\in \mathbb{N}$, such that if $s<s^*$ and $\alpha_s\le 2\alpha_{2s/3}$ then $\P ( \Cc_{\bar B(2s)} )  > a_{2}\cdot c_{0}^{k_{2}}$.
 \end{lemma}

The proof of Lemma \ref{lem:Tassion 2.2} is similar to the proof of \cite[Lemma 2.2]{Tas16}, with certain modifications needed to handle the spherical geometry in the case $\mathbb{X} = \mathbb{S}^2$; here we only give a sketch of the argument while explaining in detail the necessary modifications.

\begin{proof}
We consider separately two cases, ~$\alpha_s=s/4$ and $\alpha_s=\phi^{-1}_s(c_0/4)<s/4$, beginning with with the first case. In light of (P1) from Lemma \ref{lem:Tassion 2.1}, we have a lower bound on $\P\brb{\Xc_s(s/4)}$ of the form  $a_{2}\cdot c_{0}^{k_{2}}$. Hence, by the FKG property, it suffices to construct a finite collection of $s \times s$ squares~$S_i$ such that if $\Xc_s(s/4)$ holds for each~$S_i$ then so does $\Cc_{\bar B(2s)}$.

In what follows we refer to the labelling in Figure \ref{fig:gluing two x crossings}, which illustrates the argument in the spherical case. Consider the $s \times s$ square $ABCD$ and its translation $A'B'C'D'$ by $s/2$ along the geodesic $AB$. Observe that if the event $\Xc_s(s/4)$ holds for both squares $ABCD$ and $A'B'C'D'$ then there is a crossing of $\Sc$ inside the union of the squares that intersects the two sub-intervals of the geodesic $AB'$ formed by removing a centred interval of length~$s$. Repeating this construction along the top edge of $\bar B(2s)$ we obtain a crossing of $\bar B(2s)$ using only $X$-crossings of $s \times s$ squares.

We turn to the second case. Since $\alpha_s\le  2\alpha_{2s/3}$ and in light of Lemma \ref{lem:Tassion 2.1}, in this case we have lower bounds on both $\P\brb{\Xc_{2s/3}(\alpha_{2s/3} )} $ and $\P\brb{\Hc_s(0, 2 \alpha_{2s/3} )}$ of the form  $a_{2}\cdot c_{0}^{k_{2}}$. Hence, by the FKG property, it suffices to construct a finite collection of  $s \times s$ squares $S_i$, and $(2s/3) \times (2s/3)$ squares $T_i$, such that if $\Hc_s(0, 2 \alpha_{2s/3})$ and $\Xc_{2s/3}(\alpha_{2s/3} )$ holds for each $S_i$ and $T_i$ respectively, then so does $\Cc_{\bar B(s)}$.

In what follows we refer to the labelling in Figures \ref{fig:joining two crossings}--\ref{fig: two x crossings}, which illustrate the argument in the spherical case. Consider the $s \times s$ square $ABCD$ and its translation $A'B'C'D'$ by a distance~$d$ (to be determined) along the geodesic joining the mid-points of the sides $AD$ and~$BC$. Our aim is to deduce a horizontal crossing of the union of these squares (i.e.\ between $AD$ and $B'C'$) by assuming $\Hc_s(0, 2 \alpha_{2s/3} )$ holds for both the squares, and assuming also $\Xc_{2s/3}(\alpha_{2s/3} )$ holds for two suitably chosen $(2s/3) \times (2s/3)$ squares. In the planar case \cite{Tas16} we may let $d = 4s/3$, since then the shaded region in Figure \ref{fig:joining two crossings} forms a $(2s/3) \times (2s/3)$ square which is sufficient for this purpose. In the spherical case this shaded region is not a square for any choice of translation distance, and so we shall need a slightly different construction.

We consider the $(2s/3) \times (2s/3)$ square $abcd$ such that its `right' side $bc$ lies on $BC$ with its mid-point coinciding with the mid-point of the marked thick interval $fg$ of length $2\alpha_{2s/3}$ (see Figure \ref{fig:joining two crossings new square}). Note that $bc$ is a subset of $BC$ since $\alpha_s \le s/4$ for each $s$, and hence $s/3 + \alpha_{2s/3}$ is at most $s/2$. Once this square is fixed, we consider the unique geodesic which passes through the middle of the side $ad$ of the small square $abcd$ and orthogonal to the geodesic connecting mid-points of $AD$ and $BC$. We define the second $s\times s$ square $A'B'C'D'$ to be the square such that its left side is on this geodesic. The second $(2s/3) \times (2s/3)$ square $a'b'c'd'$ (not shown in Figure \ref{fig:joining two crossings new square}, but magnified in Figure~\ref{fig: two x crossings}) is constructed as the symmetric image of $abcd$ and its left side is on $A'D'$. Observe that the mid-point of $ad$ (marked by a dot) will lie on the marked interval~$e'h'$. We also notice, that since the distance between the mid-points of an $(2s/3) \times (2s/3)$ square is $2s/3+O(s^3)$ where $O(s^3)$ term depends on $s$ only, the square $A'B'C'D'$ is a copy of $ABCD$ shifted by $d = s/3+O(s^3)$. In particular, there is $s^*$ such that for all $s<s^*$ it holds that $d \in ( s/4, s/2)$.

\begin{figure}[t]
\centering
\includegraphics[width=0.48\textwidth]{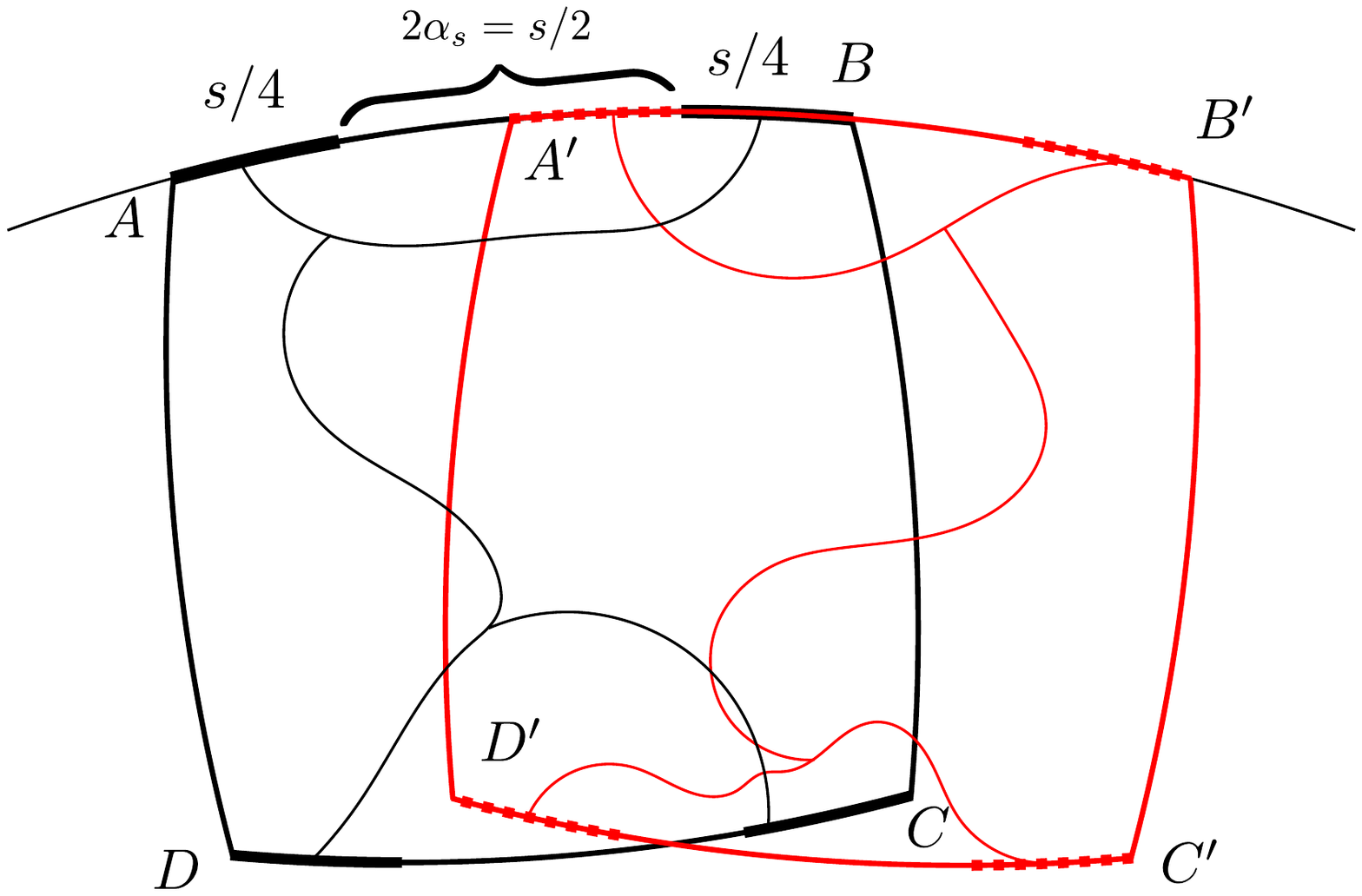}
\includegraphics[width=0.48\textwidth]{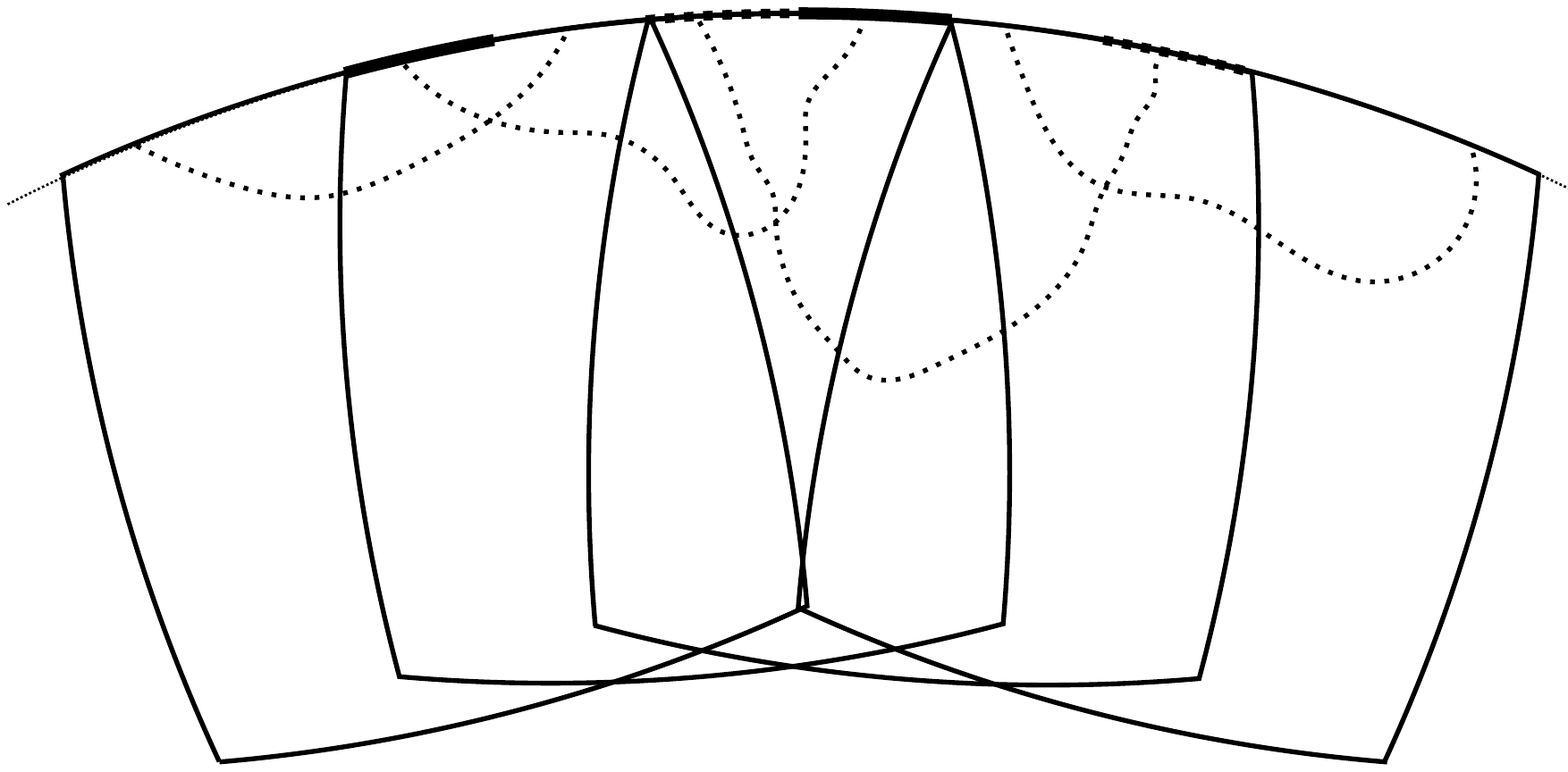}
\caption{Left: The red $s \times s$ square $A'B'C'D'$ is a translation of the black square $ABCD$ by $s/2$ along the side $AB$. If the event $\Xc_s(s/4)$ holds in both of these squares (indicated by the black/red connecting sets), then there is a connection between the top-left and top-right ends of $AB'$. Right: Repeating the construction we obtain an arbitrary long crossing which is $s$-close to a given geodesics.}
\label{fig:gluing two x crossings}
\end{figure}

\begin{figure}[h]
\centering
\includegraphics[width=0.5\textwidth]{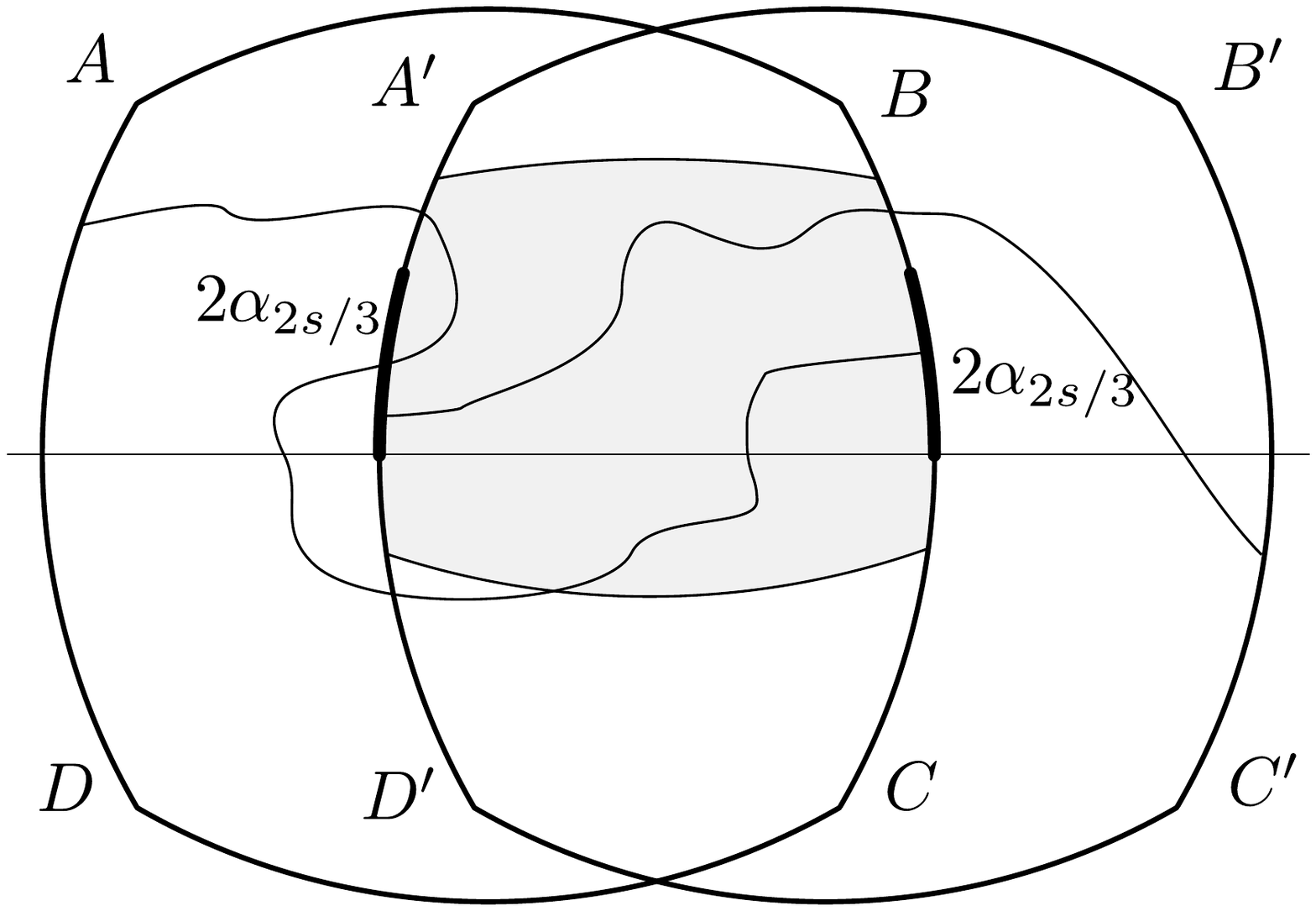}
\caption{An $s\times s$ square $ABCD$ and its translation $A'B'C'D'$ along the line joining the mid-points of $AD$ and $BC$ by a distance $d > 0$. In the planar case, if $d=4s/3$ then the shaded area is a $(2s/3) \times (2s/3)$ square; this construction was used in \cite{Tas16} to connect horizontal crossings of the $s \times s$ squares via an $X$-crossing of the shaded region. In the spherical case, the shaded area is not a square for any choice of $d$, so we use a different construction to connect horizontal crossings of the large squares.}
\label{fig:joining two crossings}
\end{figure}

\begin{figure}[h]
\centering
\includegraphics[width=0.5\textwidth]{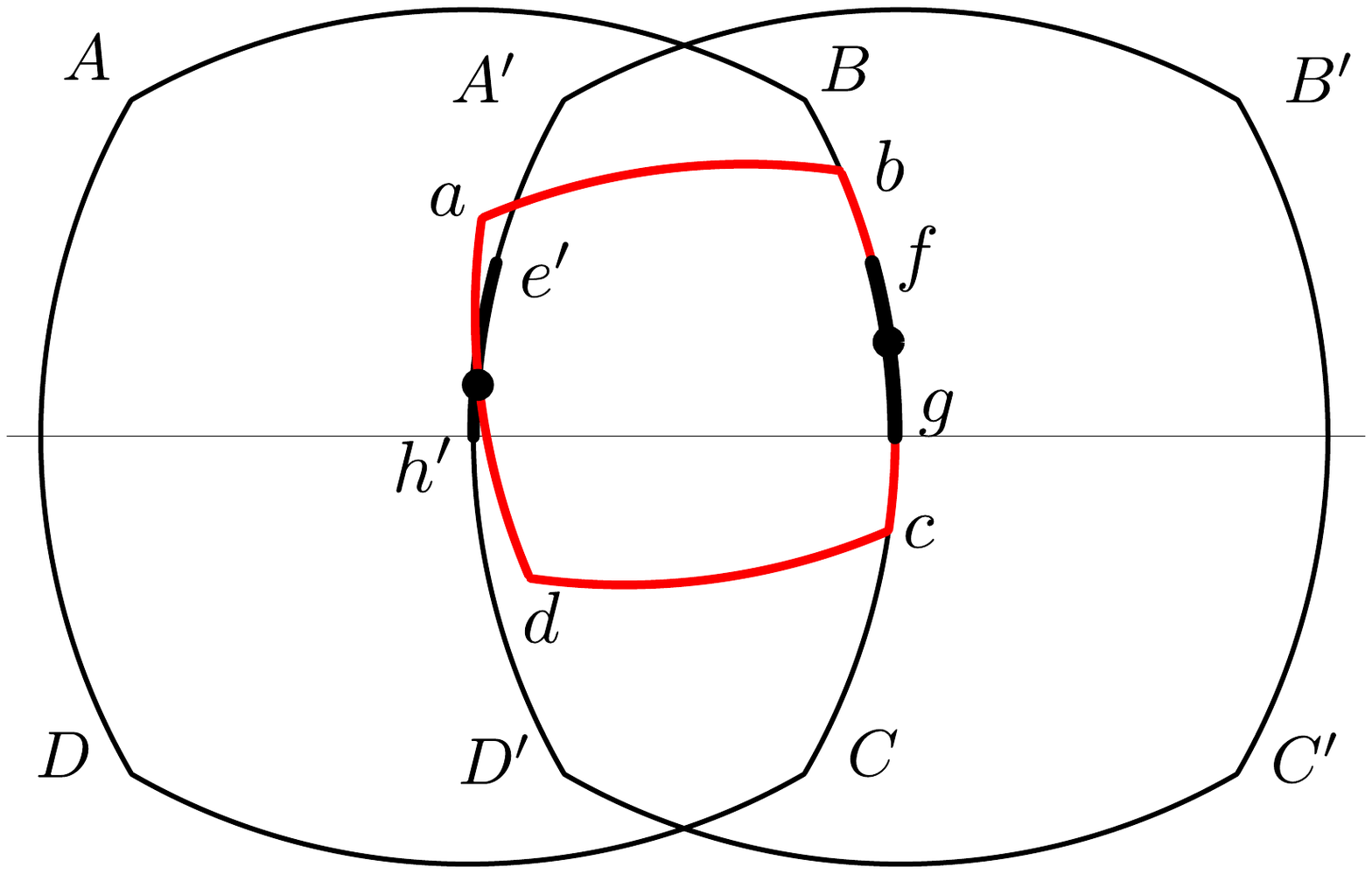}
\caption{After the small square $abcd$ is fixed, the translation distance is chosen such that the the left side of the large square $A'B'C'D'$ intersects the mid-point of $ad$.}
\label{fig:joining two crossings new square}
\end{figure}

Next let us consider the two small squares $abcd$ and $a'b'c'd'$, shown in more detail in Figure~\ref{fig: two x crossings}. We mark the middle parts of length $2\alpha_{2s/3}$ on `vertical' sides of both small squares. Two of these marked intervals $fg$ and $e'h'$ are the marked intervals on Figures \ref{fig:joining two crossings} and \ref{fig:joining two crossings new square}. As mentioned above, the intervals $eh$ and $e'h'$ intersect. By symmetry, the intervals $fg$ and $f'g'$ intersect as well. This implies that  any  curve in $abcd$ connecting $ae$ to $cg$ must intersect $a'h'$ and $c'f'$ and thus disconnect $h'd'$ from $b'f'$ inside $a'b'c'd'$. This shows that if $\Xc_{2s/3}(\alpha_{2s/3})$ holds for both squares, then the connecting curves must intersect.

We also notice that a curve connecting $a'e'$ with $h'd'$ separates $e'h'$ from the right side of the $s \times s$ square $A'B'C'D'$. Similarly, a curve connecting $bf$ and $cg$ separates $fg$ from $AD$. This implies that in the event that there are crossings from $AD$ to $fg$, from $e'h'$ to $B'C'$, and two X-crossings in $abcd$ and $a'b'c'd'$ there is a crossing connecting $AD$ to $B'C'$.

All in all, we infer a horizontal crossing of the union of the $s \times s$ squares (i.e.\ between $AD$ and $B'C'$) that are translated a distance $d \in (s/4, s/2)$ apart. We finish the proof of Lemma~\ref{lem:Tassion 2.2} by using a similar construction to the one in the proof of Lemma \ref{l:cor13}, using multiple copies of such a crossing (i.e.\ alternating `horizontally' and `vertically') and as long as $s^\ast$ is sufficiently small, to infer a crossing of $\bar B(2s)$.
\end{proof}

\begin{figure}
\centering
\includegraphics[width=0.5\textwidth]{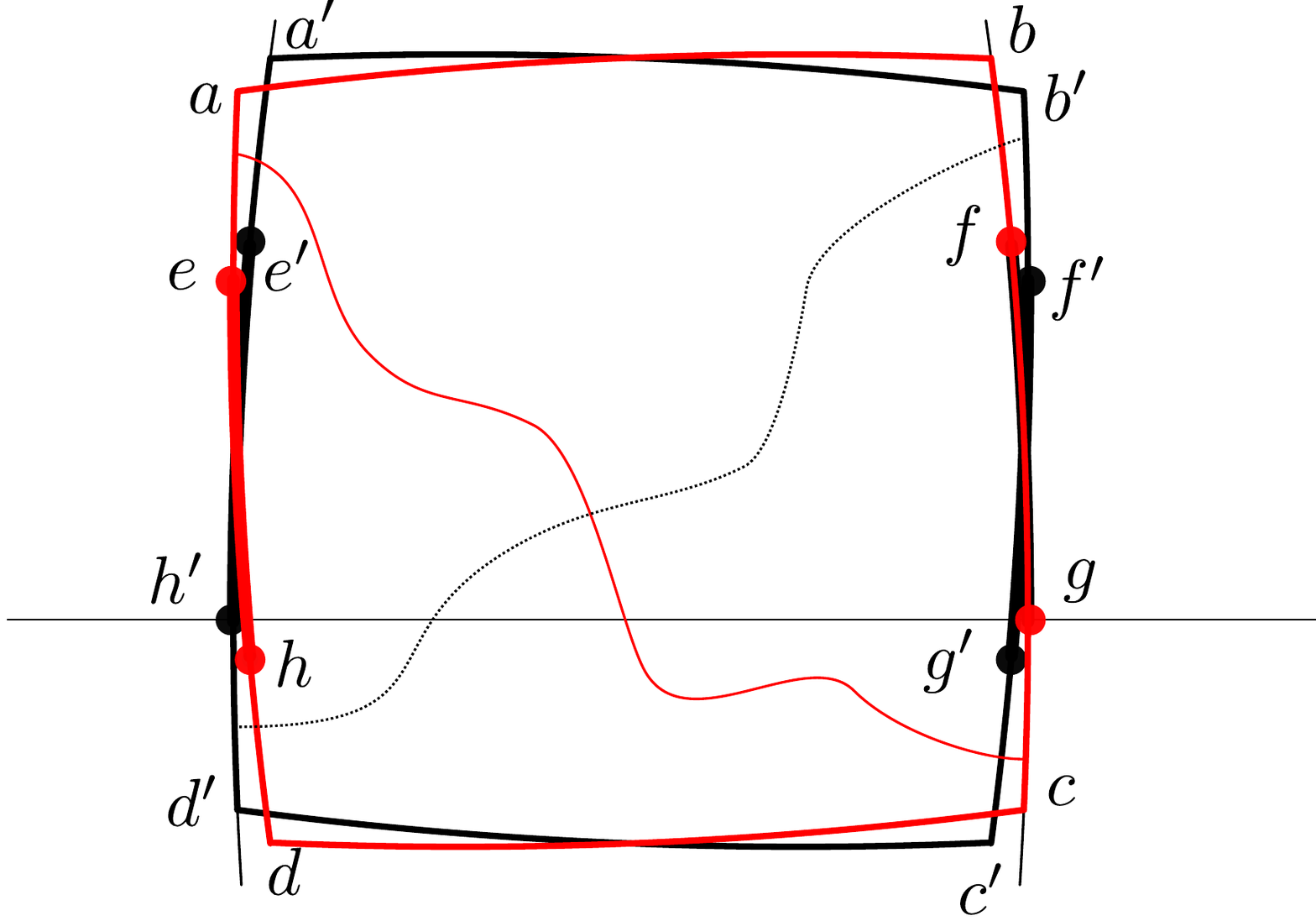}
\caption{Two $(2s/3) \times (2s/3)$ squares with marked intervals. The right side of the red square $abcd$ is on the right side of the left $s\times s$ square $ABCD$, the left side of the black square is on the left side of the right square $A'B'C'D'$. Any curve connecting $ae$ with $cg'$ inside the red square $abcd$ must intersect a curve connecting $h'd'$ with $b'f'$ inside the black square $a'b'c'd'$.
}
\label{fig: two x crossings}
\end{figure}

\begin{lemma}[Good scales imply annular crossings on larger scales; c.f. {\cite[Lemma 3.1]{Tas16}}]
\label{lem:Tassion 3.1}
There exists a sufficiently small $s^* > 0$ and absolute numbers $a_{3}>0$ and $k_{3}\in \mathbb{N}$, such that if, for some $s$ and $t$ such that $12 s \le t < s^*$, $\alpha_s \le 2\alpha_{2s/3}$ and $\alpha_t \le s$ both hold, then $\P(\Cc_{A_{t, 6t}}  ) > a_{3}\cdot c_{0}^{k_{3}}$.
\end{lemma}

\begin{proof}
Since $\alpha_s \le 2\alpha_{2s/3}$, Lemma \ref{lem:Tassion 2.2} yields a lower bound on $\P[\Cc_{\bar B(2s)}] $ of the form $a_{3}\cdot c_{0}^{k_{3}}$. By Lemma \ref{l:cor13} we then conclude the same for $\Cc_{A_{2s, 12s}}$. Since also $\P\brb{\Hc_t(0, s )} \ge c_0 / 4$ (implied by $\alpha_t \le s$), by the FKG property it suffices to find two $t \times t$ squares $S_1$ and $S_2$ such that if $\Hc_t(0, s)$ holds for $S_1$ and $S_2$, and $\Cc_{A_{2s, 12s}}$ also holds, then we may deduce $\Cc_{\bar B(t)}$.

The proof is identical to \cite[Lemma 3.1]{Tas16}, and we only briefly sketch it. Consider the two $t \times t$ squares $S_1 = (D_1; \gamma_1, \gamma')$ and $S_2 = (D_2; \gamma_2, \gamma')$ whose common side $\gamma'$ has the origin as its mid-point. Observe that if $\Cc_{A_{2s, 12s}}$ holds simultaneously with the event $\Hc_t(0, s )$ for $S_1$ and~$S_2$, then there exists a crossing of $S_1 \cup S_2$; for this observe that the distance between mid-points of a $s \times s$ square is at least $s$ (in the spherical case it is precisely $\arccos(1-2\tan^2(s/2)) = s+\frac{1}{8}s^3+O(s^5)> s$), and so the line-segment of length $s$ in the definition of $\Hc_t(0, s )$ lies inside the inner square bounding $A_{2s, 12s}$. Since such a crossing of $S_1 \cup S_2$ also implies a crossing of two squares that are translated by any smaller amount along the geodesic joining the mid-points of the opposite sides of $S_1$ and $S_2$, we infer a crossing of~$\bar B (t)$. Finally, from Lemma \ref{l:cor13} we deduce the statement.
\end{proof}

To state the final lemma in Tassion's argument, we need a certain assumption that is related to condition (6) of Theorem \ref{t:rsw}, stated for a fixed $n \in \mathbb{N}$.

\begin{assumption}
\label{a:rsw}
For a quadruple $(c, \varepsilon, C, s)$ with $c > 0, \varepsilon \in (0, 1)$, $C \ge 1$ and $s > 0$, we assume that the following holds: If
\[   \inf_{A \in \rm{Ann}_{X; C, 6}(s)}  \P(\Cc_A(\Sc_n)  ) > c  , \]
then, for each $s \times Cs$ annulus $A \subseteq X$,
\[ \P( \Cc_A)>1- \varepsilon. \]
\end{assumption}

It is clear that if Assumption \ref{a:rsw} is valid for a quadruple $(c, \varepsilon, C, s)$, then it is also valid for the quadruple $(c', \varepsilon', C, s)$ for any $c' > c$ and $\varepsilon' > \varepsilon$. For the final two lemmas, we let $a_3$ and~$k_3$ be the constants proscribed by Lemma \ref{lem:Tassion 3.1} and fix
\begin{equation}
\label{eq:c3 def a3k3}
c_3 = a_3 \cdot c_0^{k_3}.
\end{equation}

\begin{lemma}[Good scales imply larger good scales; c.f. {\cite[Lemma 3.2]{Tas16}}]
\label{lem:Tassion 3.2}
Fix $C > 1$. Then there exist a number $C_1 > 12$, depending only on $C$, and a sufficiently small $s^* > 0$, such that if $s <s^*$, $\alpha_s \le 2\alpha_{2s/3}$ and Assumption \ref{a:rsw} holds for the quadruple $(c_3,c_0/ 8, C, 12s)$, then there exists a number $t\in [12 s, C_1 s]$ such that $\alpha_t \le 2\alpha_{2t/3}$.
 \end{lemma}

\begin{proof}
Suppose $s < s^*$ and $\alpha_s \le 2\alpha_{2s/3}$. The first step is to show that $\alpha_{t_i} > s$ for at least one of $t_1 = 12 s$ or
\begin{equation}
\label{eq:t2 def}
t_2 = 2 C t_1 = 24 C s.
\end{equation}
Let $A$ be a $t_1 \times 12 C t_1$ annulus. Arguing by contradiction, if $\alpha_{t_1} \le s$, then from Lemma \ref{lem:Tassion 3.1} we deduce that
\[  \inf_{A \in \rm{Ann}_{X; C; 6}(t_1) }  \mathbb{P}( \Cc_A ) > c_3 .\]
Hence, since we make Assumption \ref{a:rsw} for the quadruple $(c_3,c_0/8, C, t_1)$, it holds that $$\P( \Cc_{A} )>1-c_0/8.$$

On the other hand, let $S$ be a $t_{2}\times t_{2}$ square whose centre coincides with the centre of~$A$ and such that one side of $S$ lies on a geodesic bisecting~$A$. Define the event $E=\Hc_{t_2}(0,s)\setminus \Hc_{t_2}(s, t_2/2)$ for square $S$, and remark that the occurrence of the event $E$ implies that $\Cc_{A}$ does not occur (see Figure \ref{fig: crossing and annulus}). Recalling the definition \eqref{eq:phisa def} of $\phi_{s}(\alpha)$ it is clear that $\P(E)\ge \phi_{t_2}(s)$. If also $\alpha_{t_2}\le s$, this implies $\alpha_{t_{2}} < t_{2}/4$ by \eqref{eq:t2 def}, and
 by (P2) of Lemma~\ref{lem:Tassion 2.1} we have $\phi_{t_2}(s) \ge c_0/4$. Since $E \subseteq \Cc_{A}^c$, this shows that $\P( \Cc_{A})$ is at most $1 - c_0/4$, which is a contradiction.

To conclude the proof of Lemma \ref{lem:Tassion 3.2}, recall that $\alpha_s$ is sub-linear in the sense that $\alpha_s < s$ for all $s > 0$. Hence if $\alpha_{t_i} > s$ for at least one of $t_1 = 12s$ or $t_2 =24 C s$, then there exist sufficiently large $C_1$, depending only on $C$, such that  $\alpha_t \ge  2\alpha_{2t/3}$ for at least one $t\in [12s, C_1 s]$.
\end{proof}

\begin{figure}
\centering
\includegraphics[width=0.5\textwidth]{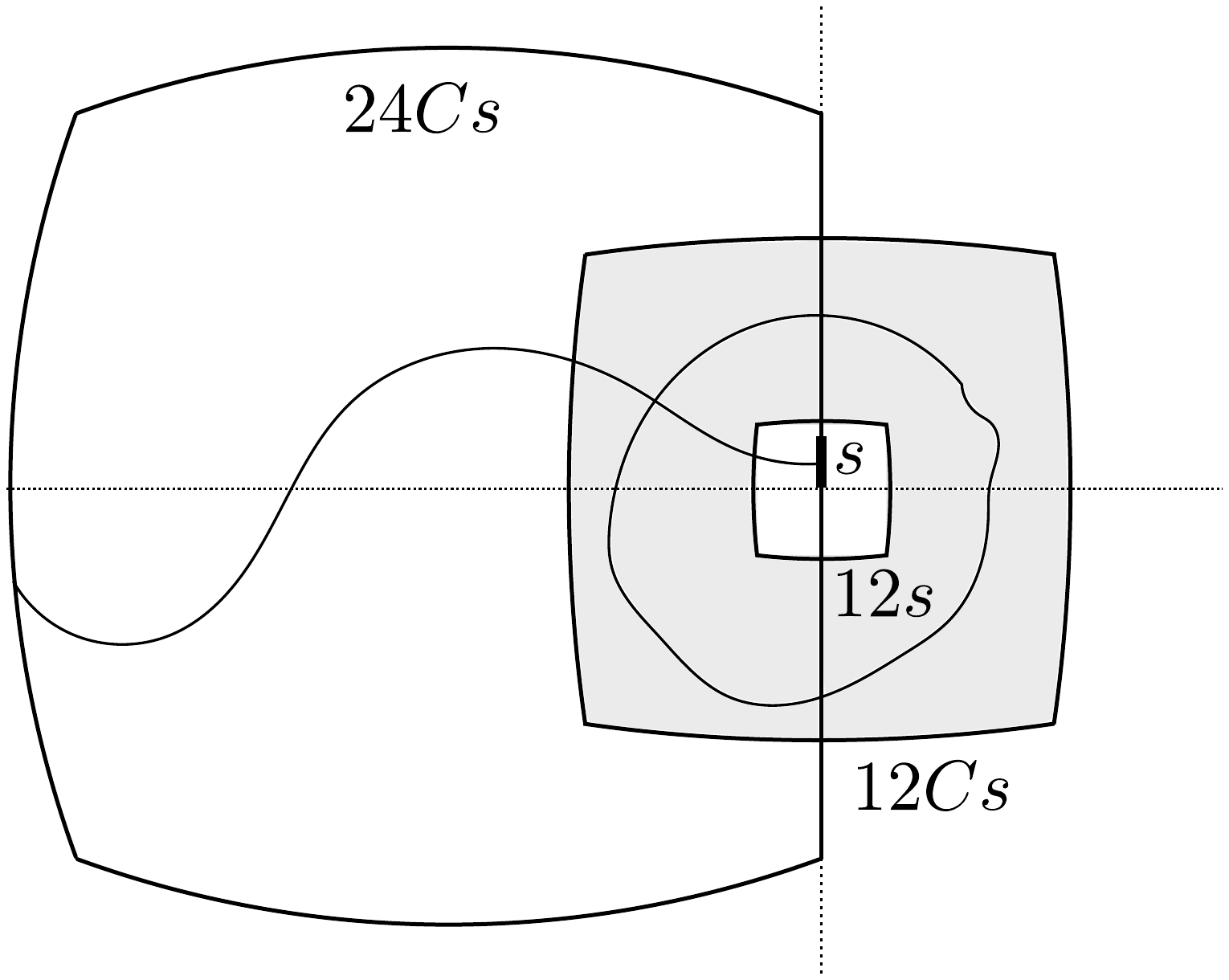}
\caption{A crossing from the left side of the white square to a small interval of length $s$ on the right side (event $\mathcal{A}$) intersects with a circuit in the grey annulus (event $\mathcal{B}$), which implies a crossing from the left side of the white square to the part of the right side lying above the interval (event $\mathcal{C}$). Hence, if the event $\mathcal{A} \setminus \mathcal{C}$ occurs, then the event $\mathcal{B}$ does not. The labels refer to the sizes of the squares in the proof of Lemma \ref{lem:Tassion 3.2}. }
\label{fig: crossing and annulus}
\end{figure}

To conclude this section we combine the preceding lemmas into a form most convenient for completing the proof of Theorem \ref{t:rsw}.

\begin{corollary}
\label{c:rsw}
Fix the constants $C > 0$ and $\bar c_1, \bar c_2 > 0$. Then there exists a sufficiently small $s^* > 0$ and numbers $a_{4}>0$, $k_{4}\in \mathbb{N}$ and $C_1 > 0$, depending only on $C$, $\bar c_1$ and $\bar c_2$, such that, if $s <s^*$, $\alpha_s > \bar c_1 s$, and Assumption \ref{a:rsw} holds for the quadruple $( c_3, c_0/ 8, C, t)$ for all $t > s$, then
   \[  \inf_{s' > C_1 s} \inf_{ B \in \rm{Box}_{X; \bar c_2}( s')  } \mathbb{P}(\Cc_B)  > a_4 \cdot c_0^{k_4}  . \]
\end{corollary}

\begin{proof}
In light of lemmas \ref{l:cor13} and \ref{lem:Tassion 2.2}, it suffices to exhibit constants $C_1, C_2 > 0$, depending only on  $C$ and $\bar c_1$, and a sequence of `good' scales $\{s^{(i)}\}_{1 \le i \le k}$ such that
\[ s^{(1)} < C_1 s/ 6 \ , \quad 12 \le  s^{(i+1)} / s^{(i)}  \le C_2  \quad \text{and} \quad  12 s^{(k)} \le s^*  ,\]
and such that $\alpha_{s^{(i)} } \le 2\alpha_{2s^{(i)} /3}$  holds for each $1 \le i \le k$.
We argue by induction. For the base case, we argue as in the proof of Lemma \ref{lem:Tassion 3.2}: since $\alpha_s > \bar c_1 s$ and $\alpha_s$ is sub-linear (in the sense that $\alpha_s \le s/4$ for all $s$), there exists a sufficiently large $C_1$, depending only on $C$ and $\bar c_1$, such that $\alpha_t \le  2\alpha_{2t/3}$ for at least one $t\in [12s, C_1 s/6]$.  Next suppose we have a scale $s^{(i)}$ such that $  s^{(i)} < s^*$ and $\alpha_{s^{(i)} } \le  2\alpha_{2s^{(i)} /3}$. We may suppose that Assumption \ref{a:rsw} holds for the quadruple $(c_3, c_0/ 8, C, 12 s^{(i)} )$. Hence by Lemma \ref{lem:Tassion 3.2} there exists a number $t \in [12  s^{(i)} , 12 C  s^{(i)} ]$ such that $\alpha_t \le 2\alpha_{2t/3}$, which concludes the induction step, and thus also Corollary \ref{c:rsw}.
\end{proof}

\subsection{Concluding the proof of Theorem \ref{t:rsw}}

Fix $s^* > 0$ to be sufficiently small such that the conclusions of Lemma \ref{l:cor13} and Corollary \ref{c:rsw} are valid. Before continuing, we discuss the roles of conditions (1) and (6) of Theorem \ref{t:rsw} in ensuring that the conclusion of Corollary~\ref{c:rsw} holds on all necessary scales and is uniform for sufficiently large $n$.

We first claim that \eqref{e:nondegenrsw} in condition (1) implies that, for each $C > 0$, $\alpha_{C s_n} / s_n$ is uniformly bounded from below. In fact, we prove the stronger statement that $\alpha_{C s_n} > r s_n$ for any $r \in (0, C/4)$ such that,
\begin{equation}
\label{e:rsn}
   \mathbb{P}(   \mathcal{L}_v(r s_n)  \cap \partial S  = \varnothing  )  > 1 - c_0 / 4
   \end{equation}
for all directions $v$ in the spherical case (resp.\ $x$ and $y$ directions in the toral case);
the existence of a single such $r > 0$ for $n$ sufficiently large is then guaranteed by \eqref{e:nondegenrsw}. Similarly to the proof of Lemma \ref{lem:Tassion 3.2}, consider the event $$E=\Hc_{C s_n}(0,r s_n) \setminus \Hc_{C s_n}(r s_n, C s_n)$$ corresponding to a $Cs_n \times Cs_n$ square $S$, and let $L$ denote the line-segment of length $r s_n$ on the boundary of $S$ used to define the event $\Hc_{C s_n}(0,r s_n)$. It is then clear that $\P(E) \ge \phi_{C s_n}(r s_n)$. If we now assume, for contradiction, that \eqref{e:rsn} holds and $\alpha_{C s_n} \le r s_n$, then since $r < C/4$, by (P2) of Lemma~\ref{lem:Tassion 2.1} it must be true that $\phi_{C s_n}(r s_n)  \ge c_0/4$. Since $E$ implies that $\partial S$ intersects~$L$, we have that
\[ \mathbb{P}(   \mathcal{L}_v(r s_n)  \cap \partial S  = \varnothing  ) =  \mathbb{P}( | L \cap \partial S| = \varnothing ) \le 1-  c_0/4 , \]
which is a contradiction.

Next we observe that condition (6) of Theorem \ref{t:rsw} implies Assumption \ref{a:rsw} on all necessary scales. To see why note that condition (6) guarantees the existence, for any choice of $c > 0$ and $\varepsilon > 0$, of constants $C_1, C_2 > 1$  such that, for all sufficiently large $n$ and all $s > C_1 s_n$, Assumption \ref{a:rsw} holds for the quadruple $(c, \varepsilon, C_2, s)$; in particular it also holds for any larger~$c$ and $\varepsilon$ (see the remark immediately after Assumption \ref{a:rsw}). Recall now that
\[  c_0(n) = \inf_{s >0 } \inf_{ B \in \rm{Box}_{X; 1}(s) } \, \mathbb{P}( \mathcal{C}_{B}(\mathcal{S}_n) ) , \]
is bounded from below by some constant $\hat{c}_0$ for sufficiently large $n$; hence, by \eqref{eq:c3 def a3k3},
the same is true for the number $c_3(n)$
prescribed by Lemma \ref{lem:Tassion 3.1}, monotonically increasing in $c_0$. Putting this together, condition (6) guarantees the existence of $C_1, C_2> 0$ such that, for all sufficiently large $n$, Assumption \ref{a:rsw} holds for the quadruple $(c_3(n), c_0(n) / 8, C_2, s)$ for all
$s > C_1 s_n$. At this point we may fix such $C_1, C_2> 0$ and $n$ sufficiently large such that the assumption holds for all $s > C s_n$.

We can now finish the proof of Theorem \ref{t:rsw}. Choose $c > 0$ as in the statement of the RSW estimates. Given the definition of $\rm{Unif}_{X;c}(s)$, and since the FKG property is valid in $X$, it is sufficient to show the existence of a constant $c_1$ such that for sufficiently large $n$,
\begin{equation}
\label{e:rswcon}
  \inf_{s > 0} \,  \inf_{k \in (0, c)} \  \inf_{B \text{ a } s \times ks \text{ box}} \,  \mathbb{P}(\Cc_B(\Sc_n )) > c_1  .
  \end{equation}
In turn, it is sufficient to establish \eqref{e:rswcon} on both the microscopic scales $s \approx s_n$, and then for all larger scales $s \gg s_n$.

For the microscopic scales $s \approx s_n$, recall that, by condition (4) of Theorem \ref{t:rsw}, there exist numbers $\delta>0$ and $c_2>0$ such that, for all sufficiently large $n$,
\begin{equation*}
 \inf_{s < \delta s_n} \,  \inf_{k \in (0, c)} \  \inf_{B \text{ a } s \times ks \text{ box}} \,  \mathbb{P}(\Cc_B(\Sc_n )) > c_2.
  \end{equation*}
By Lemma \ref{l:cor13}, the same conclusion holds for $\delta$ replaced by any constant $C$, i.e.\ there exists a $c_4$, depending on $C$, such that for sufficiently large $n$,
\begin{equation}
\label{e:rswcon2}
 \inf_{s < C s_n} \,  \inf_{k \in (0, c)} \  \inf_{B \text{ a } s \times ks \text{ box}} \,  \mathbb{P}(\Cc_B(\Sc_n )) > c_4 .
 \end{equation}

For the larger scales $s \gg s_n$, take the constant $C_1$ that was fixed above, and recall that $\alpha_{C_1 s_n} / s_n$ is uniformly bound below by some constant $\bar c_1$. Since also Assumption \ref{a:rsw} holds for the quadruple $(c_3(n), c_0(n) / 8, C_2, t)$ for all $ t > C_1 s_n $, by Corollary \ref{c:rsw} there are numbers $a_{4}>0$, $k_{4}\in \mathbb{N}$ and $C_3 > 0$, depending only on $c, \bar{c}_1$, $C_1$ and $C_2$, such that
\[  \inf_{s' > C_3 s_n} \inf_{ B \in \rm{Box}_{X; c}( s')  } \mathbb{P}(\Cc_B) > a_4 \cdot c_0^{k_4}(n) >  a_4 \cdot \hat{c_0}^{k_4},\]
which establishes \eqref{e:rswcon} for $s > C_3 s_n$. Combining with \eqref{e:rswcon2} we conclude the proof.


\smallskip
\section{Perturbation analysis}
\label{s:pert}

In this section we establish the auxiliary results used in the perturbation analysis in section~\ref{s:proof}. In the first part we prove Proposition~\ref{p:meas}, showing that crossing events are determined, outside a small error event, by the signs of the field on a (deterministic) finite set of points. In the second part we prove Lemma \ref{l:comp}, which controls the effect of a perturbation on the signs of Gaussian vectors.

 \subsection{Measurability of crossing events on a finite number of points}

We use the following preliminary lemma, which bounds the probability that the nodal set crosses any (geodesic) line-segment twice. Recall that for symmetric covariance kernels we often abuse notation by writing $\kappa_n(x) = \kappa_n(0, x)$.

\begin{lemma}[Two-point estimate of nodal crossings; c.f.\ {\cite[Proposition 4.4]{BM}}]
\label{l:tp}
Let $f$ be a Gaussian random field on $\mathbb{X}$ whose covariance kernel $\kappa$ is $C^{4}$ and is symmetric in the sense of Definition \ref{a:symmetry}. Suppose
that there exists $\delta >0$ such that, for every $x,y\in \Xb$ with $0<d(x,y)<\delta$
the random vector $(f(x),f(y))\in\R^{2}$ is non-degenerate. Define
\[ L_2 =  \sup_{v \in \mathbb{S}^1}    |\kappa''_{v}(0)|   \quad \text{and} \quad L_4 =  \sup_{v \in \mathbb{S}^1} \max_{d(0,y) < \delta} |\kappa^{(iv)}_v(y)| ,  \]
where $\kappa^{(ii)}_v$ and $\kappa^{(iv)}_v$ are the second and the fourth derivatives of $\kappa$ in direction~$v$ respectively.
Then there exists a absolute constant $c > 0$ such that, for each geodesic line-segment $\mathcal{L} \subseteq \mathbb{X}$ of length
$\varepsilon <\delta$,
\[  \mathbb{P}( | \{x \in  \mathcal{L} : f(x) = 0 \} | \ge 2 )  < c  \varepsilon^3 \sqrt{ L_2^3  + L_2^{-1} L_4^2 }   .\]
\end{lemma}

\begin{proof}
It is convenient to use the arc-length parametrisation of $\mathcal{L}$, namely let $\tilde{f} : [-\varepsilon/2, \varepsilon/2] \to \mathbb{R}$ be the restriction $f|_\mathcal{L}$ of $f$ to $\Lc$, and denote by $\tilde{\kappa}: [-\varepsilon/2, \varepsilon/2] \to \mathbb{R}$
its covariance kernel. By the symmetry assumption on $f$, the process $\tilde{f}$ is {\em stationary}, and with no loss of generality we may assume that $\tilde{f}$ is unit variance.

Let $N =  | \{x \in  \mathcal{L} : f(x) = 0 \} |$. Applying the Kac-Rice formula \cite[Theorem 6.3]{AW},
valid by the non-degeneracy assumption on $(f(x),f(y))$ in Lemma \ref{l:tp} we have
\begin{equation}
\label{eq:fact mom 2pnt corr int}
\mathbb{E}[ N(N-1) ]  = \int_{x,y \in [-\varepsilon/2, \varepsilon/2] } M_{2}(x-y) \, dxdy
\end{equation}
with $M_{2}(x)\ge 0$ the two-point correlation function of the zeros of $\tilde{f}$. It is known ~\cite{BlDi} that $M_{2}$ is given by
\begin{equation*}
M_{2}(x) = \frac{1}{\pi^{2}}\frac{-\tilde{\kappa}''(0)\cdot (1-\tilde{\kappa}(x)^{2})-\tilde{\kappa}'(x)^{2}}{(1-\tilde{\kappa}(x)^{2})^{3/2}}
\cdot \left( \sqrt{1-\rho(x)^{2}}+\rho(x)\cdot\arcsin{\rho(x)}  \right),
\end{equation*}
with $\rho$ an explicit expression in terms of $\tilde{\kappa}$ and its first two derivatives, irrelevant for our purpose. The upshot is
that the function
\[ t\mapsto \sqrt{1-t^{2}}+t\cdot\arcsin{t} \]
 is bounded from above, hence
\begin{equation}
\label{e:g}
M_{2}(x) \le c_{1} \cdot \frac{-\tilde{\kappa}''(0)\cdot (1-\tilde{\kappa}(x)^{2})-\tilde{\kappa}'(x)^{2}}{(1-\tilde{\kappa}(x)^{2})^{3/2}}
\end{equation}
for some absolute constant $c_{1}>0$.

Finally, recall that $\kappa$ is $C^{4}$, and so Taylor's theorem implies that each $x \in [0, \varepsilon]$ satisfies,
\begin{equation}
\label{eq:r der expand}
\left| \tilde{\kappa}(x) - 1 -  \frac{1}{2} \tilde{\kappa}'(0) x^2 \right| \le  \max_{ y \in B(\delta) } | \tilde{\kappa}^{(iv)}(y) | x^4 \ \  \text{and} \ \  \left| \tilde{\kappa}'(x) -  \tilde{\kappa}'(0)  x \right| \le \max_{ y \in B(\delta) } | \tilde{\kappa}^{(iv)}(y) |  x^3  .
\end{equation}
Expanding \eqref{e:g} into the Taylor polynomial of fourth degree around the origin with the help of~\eqref{eq:r der expand},
we obtain the bound
\[ M_{2}(x)   \le c_2 \left( \tilde{\kappa}'(0)^{3/2} +    \tilde{\kappa}'(0)^{-1/2}  \max_{ y \in B(\delta)} \tilde{\kappa}^{(iv)}(y) \right) |x|    \]
with some absolute constant $c_{2}>0$.
Finally, integrating the latter inequality over $x,y \in [-\varepsilon/2, \varepsilon/2]$ as in \eqref{eq:fact mom 2pnt corr int} yields that
\[ \mathbb{E}[ N(N-1) ]  < c_3  \varepsilon^3 \left( \tilde{\kappa}'(0)^{3/2} +  \tilde{\kappa}'(0)^{-1/2}   \max_{ y \in B(\delta) } \tilde{\kappa}^{(iv)}(y)  \right)   ,\] with $c_{3}>0$ absolute.
Since $\mathbb{X}$ has constant curvature, the ratio of the derivatives of $\tilde{\kappa}$ and $\kappa$ are bounded from above and from below by absolute constants, and so by Markov's inequality we conclude the proof.
\end{proof}

We now state the main implication of Lemma \ref{l:tp} in our setting. Recall the set-up of the perturbation analysis from section \ref{s:proof}, and in particular the constant $\delta_0$ and the limit kernel~$K_\infty$. The following is an easy corollary of Lemma \ref{l:tp}, the uniform convergence of~$\kappa_n$ on $B(\delta_{0})$ to $K_{\infty}$ along with its first four derivatives, and the fact that~$K_\infty$ satisfies Assumption~\ref{a:nondegen} (and so in particular has strictly-positive second derivatives at the origin); by the above we can take a single number
$\delta>0$ satisfying the assumptions of Lemma \ref{l:tp} applied to $f=f_{n}$ for $n$ sufficiently large (i.e.\ the $\delta$ corresponding to $K_{\infty}$).

\begin{corollary}
\label{c:tp}
There exists a number $0<\delta<\delta_{0}$ sufficiently small, and $c_{1}>0$ sufficiently large depending on $K_{\infty}$ only, such that for $n \in \mathbb{N}$ sufficiently large the following holds. For every geodesic line-segment $\mathcal{L} \subseteq \mathbb{X}$ of length $\ell \in (0, \delta)$,
\[  \mathbb{P}( | \{x \in  \mathcal{L} : f_n(x) = 0 \} | \ge 2 )  < c_{1}  (\ell/s_n)^3    .\]
\end{corollary}

We can now complete the proof of Proposition \ref{p:meas}. For this we will use the following notion of a `triangular decomposition' of a polygon.

\vspace{0.1cm}
\begin{definition}
~
\begin{enumerate}

\item For a polygon $P = (D; \gamma, \gamma')$ as in Definition \ref{d:poly}, a \textit{triangular decomposition} $\mathbb{T}$ of~$P$ is a (finite) embedded graph on $\mathbb{X} \cap P$ such that each edge is a geodesic line-segment, each face has three boundary edges, and the union of the faces equals $P$, save for boundaries.

\item A triangular decomposition $\mathbb{T}$ of a polygon $P$ is said to be {\em compatible} with $P$ if both $\gamma$ and $\gamma'$ can be expressed as the union of edges of $\mathbb{T}$.

\item A \textit{triangular decomposition} of an annulus $A$ as in Definition \ref{d:ann} is defined analogously.

\end{enumerate}

\end{definition}

\begin{proof}[Proof of Proposition \ref{p:meas}]
Fix $n \in \mathbb{N}$ sufficiently large, $\delta>0$ sufficiently small and $c_{1}$ sufficiently large, so that the conclusion of Corollary \ref{c:tp} holds, and fix also $c, r > 1$ as in the statement of Proposition \ref{p:meas}. Let $s > 0$, $\varepsilon \in (0, 1)$ and $Q \in \rm{Poly}_{\mathbb{X}; c}(s) \cup \rm{Ann}_{\mathbb{X}; c;r}(s)$ be given. By the definition of the sets $\rm{Poly}_{\mathbb{X}; c}(s)$ and $\rm{Ann}_{\mathbb{X}; c}(s)$, there exists a number $c_2 > 0$, depending only on $c$ and $r$, such that for each $\ell \in (0, s \wedge \delta]$ there exists a triangular decomposition $\mathbb{T}$ of~$Q$ with the following properties: (i) if $Q \in \rm{Poly}_{\mathbb{X}; c}(s)$ then $\mathbb{T}$ is compatible with $Q$; (ii) the edges of $\mathbb{T}$ have length at most $\ell s$; and (iii) $\mathbb{T}$ has at most $c_2 (s/\ell)^2$ vertices.

Fix an edge $e$ in $\mathbb{T}$ and consider the event that $e$ is crossed at least twice by the nodal set.
Applying Corollary \ref{c:tp}, there exists a constant $c_2$, depending only on $K_\infty$, such that this event is of probability at most $c_2  ( \ell / s_n)^3$. By the union bound, the event $E$ that \textit{all} the edges of $\mathbb{T}$ are crossed at most once by the nodal set has probability bounded from below by
\[  1 - c_1 c_2 (s/\ell)^2 (\ell/s_n)^3  = 1 -  c_1 c_2 s^2 s_n^{-3} \ell. \]
Setting $$\ell = \min\{\delta,\, s,\, \varepsilon s_n^3/(c_1 c_2 s^2)\},$$ this is bounded from below by $ 1 - \varepsilon$. Moreover, with this choice of $\ell$, the cardinality of~$\mathcal{P}$ is at most
\begin{equation*}
|\mathcal{P}| \le  c_2 \max\{ \delta^{-2} s^2,\,1,\,  (c_1 c_2)^2 \varepsilon^{-2}  (s/s_n)^{-6}\}
\end{equation*}
Since the sets $\rm{Poly}_{\mathbb{X}; c}(s)$ and $\rm{Ann}_{\mathbb{X}; c}(s)$ are empty unless $s$ is less than a constant ($2 \pi$ in the spherical case, $1$ in the toral case), this in turn is bounded from above by
\[   |\mathcal{P}|  \le   c_3 ( \varepsilon^{-2}  (s/s_n)^{-6} \wedge 1) ,    \]
where $c_3 > 0$ is a constant depending only on $c$, $r$, $\delta$ and $K_\infty$.

Finally, observe that on the event $E$ the crossing event $\Cc_Q(\Sc_n^+)$ is determined by the subset of edges in $\mathbb{T}$ that are crossed exactly once by the nodal set (if $Q \in \rm{Poly}_{\mathbb{X}; c}(s)$ the compatibility of $\mathbb{T}$ with~$Q$ is crucial in this step). Since this subset of edges is, in turn, determined by the signs of $f_n$ on the vertices of the triangular decomposition $\mathbb{T}$, we conclude the proof.
\end{proof}

\subsection{Proof of Lemma \ref{l:comp}}

We begin with the first statement. Define the matrices
\[   \Sigma_{Z} = n \delta \id_n \quad \text{and} \quad \Sigma_{W} = n \delta \id_n + \Sigma_Y - \Sigma_X , \]
where $\id_n$ denotes the $n\times n$ identity matrix. By the Gershgorin circle theorem and the definition of $\delta$, the matrix~$\Sigma_W$ is positive-definite. Hence
\[   Y + Z \stackrel{d}{=} X + W  \]
where $Z$ and $W$ are independent Gaussian random vectors with respective covariance matrices $\Sigma_Z$ and $\Sigma_W$.

Fix $\varepsilon > 0$ and define the events
\[  \mathcal{E}_1 =   \bigcup_{i = 1}^n\left\{   |Y_i| < \varepsilon \right\}  \, , \quad  \mathcal{E}_2 =
\bigcup_{i=1}^n \left\{   |Z_i| > \varepsilon \right\}   \, , \quad \mathcal{E}_3 =   \bigcup_{i = 1}^n\left\{   |X_i| < \varepsilon \right\}   \quad \text{and} \quad  \mathcal{E}_4 =
\bigcup_{i=1}^n \left\{   |W_i| > \varepsilon \right\} . \]
Observe that the variance of the components of $Y$ and $X$ are at least one, whereas the variance of the components of $Z$ and $W$ are at most $(n+1)\delta$. Hence by the union bound, standard results on the maximum of Gaussian vectors, and Markov's inequality, there exists an absolute number $c_1 > 0$ such that
\[ \mathbb{P}( \mathcal{E}_1 ) + \mathbb{P}( \mathcal{E}_3)  <   c_1 n \varepsilon \quad \text{and}  \quad \mathbb{P} (\mathcal{E}_2)  + \mathbb{P} (\mathcal{E}_4)  <  c_1 (\log n \vee 1)^{1/2} \varepsilon^{-1}  ((n+1)\delta)^{1/2}  . \]
This implies that we may couple the vectors $X$ and $Y$ so that, outside of an event of probability
\[   < c_1 ( n \varepsilon    +  (\log n \vee 1)^{1/2} \varepsilon^{-1}  ( (n+1) \delta)^{1/2} ) , \]
the signs of all the components of the vectors are equal, and hence all the events measurable w.r.t\ the signs of the vectors have the same probability up to the said error. To optimise the result we set
\[   \varepsilon = \delta^{1/4} (n+1)^{1/2} n^{-1/2} (\log n \vee 1)^{1/4}  ,\]
which yields the error probability as
\[   c_1  n^{1/2} (n+1)^{1/4}  (\log n \vee 1)^{1/4}  \delta^{1/4}  < c_2 \left(n^{3 + \eta} \delta \right)^{1/4}, \]
for a constant $c_2$ depending only on $\eta>0$.

For the second statement the argument is similar. Since $\Sigma_Y - \Sigma_X$ is positive-definite, one may write $Y \stackrel{d}{=} X + W$ where $W$ is an independent Gaussian random vector with covariance matrix $\Sigma_Y - \Sigma_X$. Fix $\varepsilon > 0$ and let $\mathcal{E}_1$, $\mathcal{E}_3$ and $\mathcal{E}_4$ be defined as before. Since the variance of the components of~$W$ are at most $\delta$, as before there exists an absolute $c_1 > 0$ such that
\[ \mathbb{P}( \mathcal{E}_1 ) + \mathbb{P}( \mathcal{E}_3)  <   c_1 n \varepsilon   \quad \text{and}  \quad \mathbb{P}  (\mathcal{E}_4)  <  c_1 (\log n \vee 1)^{1/2} \varepsilon^{-1}  \delta^{1/2}  . \]
Hence we may couple the vectors $X$ and $Y$ so that, outside of an event of probability
\[   < c_1 ( n \varepsilon    +  (\log n \vee 1)^{1/2} \varepsilon^{-1}   \delta^{1/2} ) , \]
the signs of all the components of the vectors are equal. Setting
\[   \varepsilon = \delta^{1/4} n^{-1/2} (\log n \vee 1)^{1/4}  ,\]
the error is at most $c_2 \left(n^{2 + \eta} \delta \right)^{1/4}$ for a constant $c_2$ depending only on $\eta>0$.

\smallskip

\bibliography{paper}{}
\bibliographystyle{plain}

\end{document}